\documentclass[12pt]{amsart}
\usepackage{amsmath,amssymb,amsfonts,latexsym,amscd,psfrag,mathabx,graphicx,mathrsfs,stmaryrd}
\usepackage[square,sort,comma,numbers]{natbib}
\usepackage{graphicx,tikz}
\usepackage[utf8]{inputenc}
\usepackage[T1]{fontenc}
\usepackage{lmodern}
\usepackage{epstopdf}

\usepackage{tikz}
\usepackage{verbatim}
\usetikzlibrary{shapes,shadows,calc}
\usepgflibrary{arrows}
\usetikzlibrary{arrows, decorations.markings, calc, fadings, decorations.pathreplacing, patterns, decorations.pathmorphing, positioning}
\tikzset{nodc/.style={circle,draw=blue!50,fill=pink!80,inner sep=4.2pt}}
\tikzset{nod1/.style={circle,draw=black,fill=black,inner sep=1pt}}
\tikzset{noddee/.style={circle,draw=black,fill=black,inner sep=1.6pt}}
\tikzset{nod3/.style={circle,draw=black,fill=black,inner sep=2.5pt}}
\tikzset{nodempty/.style={circle,draw=black,inner sep=2pt}}
\tikzset{nodel/.style={circle,draw=black,inner sep=2.2pt}}

\tikzset{nodinvisible/.style={circle,draw=white,inner sep=2pt}}
\tikzset{nodpale/.style={circle,draw=gray,fill=gray,inner sep=1.6pt}}

\usepackage{scalefnt}
\usetikzlibrary{arrows,decorations.pathmorphing,backgrounds,positioning,fit,petri}

\usepackage{subfigure}
\usepackage{float}
\usepackage{mathtools}
\DeclarePairedDelimiter\ceil{\lceil}{\rceil}

\theoremstyle{plain}
\newtheorem{proposition}[equation]{Proposition}

\newtheorem{theorem}[equation]{Theorem}
\newtheorem{observation}[equation]{Observation}

\newtheorem{corollary}[equation]{Corollary}
\newtheorem{lemma}[equation]{Lemma}
 
\theoremstyle{definition}
\newtheorem{problem}[equation]{Problem}
\newtheorem{definition}[equation]{Definition}
\newtheorem{remark}[equation]{Remark}
\newtheorem{example}[equation]{Example}

\newcommand{\PE}{\mathcal{P}}

\newcommand{\BE}{\mathfrak{B}}

\newcommand{\CE}{\mathcal{C}}

\newcommand{\FE}{\mathcal{F}}

\newcommand{\SE}{\mathcal{S}}

\newcommand{\UE}{\mathcal{U}}
\newcommand{\DE}{\mathcal{D}}

\newcommand{\rom}{\textnormal{R}}

\newcommand{\dist}{\operatorname{dist}}

\newcommand{\bp}{\operatorname{bp}}
\newcommand{\st}{\operatorname{st}}
\newcommand{\Max}{\operatorname{Max}}
\newcommand{\Min}{\operatorname{Min}}
\newcommand{\Mid}{\operatorname{Mid}}
\newcommand{\md}{\operatorname{mid}}
\newcommand{\comp}{\operatorname{Comp}}
\newcommand{\core}{\operatorname{Core}}
\newcommand{\red}{\operatorname{Red}_3}
\newcommand{\os}{\operatorname{os}}
\newcommand{\vdt}{\operatorname{vdt}}
\newcommand{\midd}{\operatorname{MD}}
\newcommand{\we}{\operatorname{we}}

\newcommand{\C}{$\mathcal{C}$ }
\newcommand{\V}{$\varphi$ }

\begin{document}
\title[Order-sensitive domination in posets]{Order-sensitive domination in partially ordered sets}

\author{Yusuf Civan, Zakir Deniz and Mehmet Akif Yetim}
\address{Department of Mathematics, Suleyman Demirel University,
Isparta, 32260, Turkey.}
\email{yusufcivan@sdu.edu.tr}

\address{Department of Mathematics, Duzce University, Duzce, 81620, Turkey.}
\email{zakirdeniz@duzce.edu.tr}

\address{Department of Mathematics, Suleyman Demirel University,
Isparta, 32260, Turkey.}
\email{akifyetim@sdu.edu.tr}

\date{\today}

\thanks{}


\begin{abstract}
For a (finite) partially ordered set (poset) $P$, we call a dominating set $D$ in the comparability graph of $P$, an \emph{order-sensitive dominating set} in $P$ if either $x\in D$ or else $a<x<b$ in $P$ for some $a,b\in D$ for every element $x$ in $P$ which is neither maximal nor minimal, and denote by $\gamma_{\os}(P)$, the least size of an order-sensitive dominating set of $P$. For every graph $G$ and integer $k\geq 2$, we associate to $G$ a graded poset $\mathscr{P}_k(G)$ of height $k$, and prove that
$\gamma_{\os}(\mathscr{P}_3(G))=\gamma_{\rom}(G)$ and $\gamma_{\os}(\mathscr{P}_4(G))=2\gamma(G)$ hold, where $\gamma(G)$ and $\gamma_{\rom}(G)$ are the domination  and Roman domination number of $G$ respectively. Apart from these, we introduce the notion of a \emph{Helly poset}, and prove that when $P$ is a Helly poset, the computation of order-sensitive domination number of $P$ 
can be interpreted as a weighted clique partition number of a graph, the \emph{middle graph} of $P$. Moreover, we show that the order-sensitive domination number of a poset $P$ exactly corresponds to the biclique vertex-partition number of the associated bipartite transformation of $P$. Finally, we prove that the decision problem of order-sensitive domination on posets of arbitrary height is NP-complete, which is obtained by using a reduction from EQUAL-$3$-SAT problem. 
\end{abstract} 
\maketitle

\section{Introduction}\label{sect:intro}
Domination theory is one of the well-established main streams in graph theory with various applications to real-world problems. Most domination parameters come up with mainly by imposing conditions on the sets that dominate the graph, regardless of where the graph itself derived from. As there is a growing amount of research interconnecting various fields of mathematics (algebra, topology, ect.) to graph theory, one may naturally consider an invariant of a graph to respect the substructure where the graph is constructed. 

The main purpose of our present work is to introduce and study a new variation of a domination parameter that fulfills such an expectation over the bridge between partially ordered sets (posets) and graph theory. The class of comparability graphs provides one of the most natural way to associate a graph to a given poset $P$. Recall that the comparability graph $\comp(P)$ of a poset $P$ is defined over the same ground set with $P$ having edges corresponding to comparabilities in $P$. In the graph theoretical side, the characterization of comparability graphs carried over the existence of a transitive orientation on the vertex sets. Since non-isomorphic posets may have identical comparability graphs, a graph parameter defined over comparability graphs may not recognize the role of a fixed underlying transitive orientation. At this point, we insist that a dominating set in $\comp(P)$ should further respects the structure of $P$. In more detail, we denote by $\Mid(P)$, the set of elements of $P$ which are neither maximal nor minimal in $P$. We call a dominating set $D$ in $\comp(P)$, an \emph{order-sensitive dominating set} of $P$, if $x\in D$ or there exist $a,b\in D$ such that $a<x<b$ in $P$ for every $x\in \Mid(P)$, and the \emph{order-sensitive domination number} of $P$, denote by $\gamma_{\os}(P)$, is defined to be the least size of an order-sensitive dominating set in $P$. Obviously, we have the convention that $\gamma_{\os}(P)=\gamma(\comp(P))$ whenever $\Mid(P)=\emptyset$, where $\gamma(G)$ denotes the domination number of a graph $G$. Therefore, the order-sensitive domination number is perceptible when
$\Mid(P)\neq \emptyset$. In order to distinguish this borderline, we mainly consider the family $\PE_3(k)$ of posets of height $k$ for $k\geq 3$ in which every element is contained in a chain of size at least $3$.

We prove that the order-sensitive domination number of a poset $P$ is equal to the (ordinary) domination number of the comparability graph of a poset $P^*$ constructed from $P$. Furthermore, we show that for every poset $P\in \PE_3(k)$ for some $k\geq 4$, there exists a poset $\red(P)\in \PE_3(3)$, called the \emph{height three reduction} of $P$, such that $\gamma_{\os}(P)=\gamma_{\os}(\red(P))$. 

At a first glance, such a generalization may seem to be a transfer of a notion from graph theory to the theory of posets. However, we verify that this particular new parameter has a role to play in the field of domination theory of graphs as well. For a given graph $G$, we associate to it a graded poset $\mathscr{P}_k(G)$ of height $k$ for each $k\geq 2$, and prove that $\gamma_{\os}(\mathscr{P}_3(G))=\gamma_{\rom}(G)$ and $\gamma_{\os}(\mathscr{P}_4(G))= 2\gamma(G)$, where $\gamma_{\rom}(G)$ denotes the Roman domination number of $G$. Moreover, we introduce the notion of a \emph{Helly poset}, and show that if $P\in \PE_3(k)$ is a Helly poset, then $\gamma_{\os}(P)$ can be calculated from the weighted clique partition number of a graph, the \emph{middle graph}, associated to $P$ (see Section~\ref{sec: middle} for more details). By way of application, we prove that
$\gamma_{\rom}(G)\leq 2\chi(\overline{G^2})$ holds for every $(3\text{-sun},C_4,C_5,C_6)$-free graph $G$, where $G^2$ is the square graph of $G$. 

Our next move is to establish a connection between the order-sensitive domination number of a poset $P$ and the biclique vertex-partition number of a bipartite graph $\mathcal{B}(P)$ constructed from $P$, that can be of independent interests. Recall that a biclique in a graph $G$ is a complete bipartite subgraph (not necessarily induced) of $G$. The biclique vertex-partition number $\bp(G)$ of $G$ is the least integer $d$ for which the vertex set of $G$ can be partition into $d$ bicliques of $G$. On this direction, we show the equality $\gamma_{\os}(P)=\bp(\mathcal{B}(P))$ whenever $P\in \PE_3(k)$ for some $k\geq 3$.

Finally, we determine the complexity of order-sensitive domination. In detail, we prove that for a given poset $P\in \PE_3(k)$ with $k\geq 3$ and a positive integer $d$, the problem of deciding whether there exists an order-sensitive dominating set in $P$ of size at most $d$ is NP-complete. When $k\geq 4$, the claimed result is obtained by a reduction  from EQUAL-$3$-SAT problem.

\section{Preliminaries}\label{sect:prelim}
\subsection{Graphs}
All the graphs considered in this paper are finite, simple and connected. If $G$ is a graph, $V(G)$ and $E(G)$ denote its vertex set and edge set, respectively. If $S\subset V=V(G)$, the graph induced by $S$ is written $G[S]$. We abbreviate $G[V\backslash S]$ to $G-S$. For a given vertex $x\in V(G)$, the set $N_G(x):=\{y\in V(G)\colon xy\in E(G)\}$ denote the (open) neighborhood of $x$ in $G$.

An \emph{independent set} in a graph is a set of pairwise non-adjacent vertices, while a \emph{clique} means a set of pairwise adjacent vertices. A graph is \emph{bipartite} if its vertex set can be partitioned into two independent sets. The complement $\overline{G}$ of a graph $G$ is the graph on the same vertex set $V(G)$ such that two vertices are adjacent in $\overline{G}$ if and only if they are not adjacent in $G$. The \emph{distance} between two vertices $u$ and $v$ in $G$, denoted by $dist_G(u,v)$, is the length of the shortest path connecting $u$ and $v$ in $G$. The square $G^2$ of a graph $G$ is a graph on $V(G)$ such that two vertices $u$ and $v$ are adjacent in $G^2$ if and only if $dist_G(u,v)\leq 2$. A set $S\subseteq V(G)$ is said to be a \emph{$2$-packing} in $G$ if $dist_G(u,v)>2$ for every $u,v\in S$.

Throughout the paper, $K_n$, $P_n$ and $C_k$ will denote the complete, path and cycle graphs on $n\geq 1$ and $k\geq 3$ vertices respectively. Moreover, we denote by $K_{p,q}$, the complete bipartite for any $p, q\geq 1$. A graph is said to be \emph{weakly chordal} if it does not contain any induced $C_k$ and $\overline{C_k}$ for $k\geq 5$.  A graph is called \emph{chordal bipartite} if it is both bipartite and weakly chordal.

In a graph $G$, a subset $S\subseteq V(G)$ is called a \emph{dominating set} of $G$, if any vertex which is not in $S$ is adjacent to a vertex in $S$. Furthermore, a set $D\subseteq V(G)$ is called a \emph{total dominating set} of $G$ if each vertex of $G$ is adjacent to a vertex in $D$. The minimum size of a dominating set (resp. total dominating set) of $G$, denoted by $\gamma(G)$ (resp. $\gamma_t(G)$), is called the \emph{domination number} (resp. \emph{total domination number}) of $G$. A \emph{Roman dominating function} of a graph $G$ is a function $f:V(G)\to \{0,1,2\}$ such that every vertex $v$ with $f(v)=0$ has a neighbor $u$ with $f(u)=2$ in $G$. The \emph{weight} of a Roman dominating function $f$ is the value $\we(f)=\sum_{v\in V(G)}f(v)$. The minimum weight of a Roman dominating function of a graph $G$ is called the \emph{Roman domination number} of $G$, denoted by $\gamma_{\rom}(G)$.

Throughout, we use the notation $[n]:=\{1,2,\ldots,n\}$ for an integer $n\geq 1$.

A $k$-\emph{coloring} of a graph $G$ is a mapping $\kappa\colon V(G)\to [k]$ such that $\kappa(x)\neq \kappa(y)$ for every edge $xy\in E(G)$. The \emph{chromatic number} $\chi(G)$ of $G$ is the least integer $k$ such that $G$ admits a $k$-coloring.

A \emph{biclique} in a graph is a (not necessarily induced) subgraph isomorphic to a complete bipartite graph. A set $\BE=\{B_1, B_2,\ldots,B_k\}$ of bicliques of a graph $G$ is a \emph{biclique vertex-cover} of $G$ with size $k$, if each vertex of $G$ belongs to at least one biclique in $\BE$. A biclique vertex-cover $\BE=\{B_1,B_2, \ldots,B_k\}$ is said to be a \emph{biclique vertex-partition} if bicliques in $G$ are pairwise disjoint, that is, each vertex of $G$ belongs to exacly one biclique in $\BE$. Biclique vertex-partition number of a graph $G$, denoted by $\bp(G)$, is defined to be the least integer $k$ such that $G$ admits a biclique vertex-partition of size $k$. 

\begin{remark} It is known from \citep{FMPS} that a graph has a biclique vertex-cover of size at most $k$ if and only if it has a biclique vertex-partition of size $k$. Therefore we do not distinguish between the biclique vertex-cover and biclique vertex-partition of a graph $G$ and use the notation $\bp(G)$, although we mostly appeal to biclique vertex-coverings of graphs throughout Section~\ref{sec:biclique}. 
\end{remark}

\subsection{Posets} A \emph{partially ordered set} (\emph{poset}, in short) is a pair $P=(X,\leq_P)$, where $X=X(P)$ is a set, and $\leq_P$ is a partial order, which is a reflexive, antisymmetric and transitive binary relation on $X$. We write $x\leq_P y$ (or $y\geq_P x$) when $(x,y) \in P$, and accordingly  $x<_P y$ when  $x\leq_P y$ and $x\neq y$. We further say that $y$ \emph{covers} $x$ and denote $x\prec_Py$ if $x<_P y$ and there is no element $z\in X$ such that $x<_P z <_P y$. If $x\leq_P y$ or $y\leq_P x$ for some $x,y\in X$, then $x$ and $y$ are said to be \emph{comparable} in $P$; otherwise they are \emph{incomparable} in $P$, denoted by $x\parallel_P y$. When there is no confusion, we drop the subscript $P$.

An element $x\in X$ is called a \emph{maximal element} (resp., \emph{minimal element}) of $P$ if there is no element $y\in X$  with $x<_P y$ (resp., $y<_P x$). We denote by $\Max(P)$ and $\Min(P)$ the set of maximal and minimal elements of $P$, respectively. Furthermore, we let $\Mid(P):=X\setminus (\Max(P)\cup \Min(P))$, the set of elements in $P$ which are neither maximal nor minimal. For a given subset $S\subseteq X$, the set $D_P(S):=\{x\in X\colon x<s\;\textnormal{for some}\;s\in S\}$ is the \emph{down-set} of $S$ in $P$, while $U_P(S):=\{x\in X \colon x>s\;\textnormal{for some}\;s\in S\}$ is the \emph{up-set} of $S$ in $P$. For a given subset $S\subseteq X$, we also define the \emph{common down-set} and the \emph{common up-set} of $S$ in $P$ by $CD_P(S):=\{x\in X\colon x<s\;\textnormal{for all}\;s\in S\}$ and $CU_P(S):=\{x\in X \colon x>s\;\textnormal{for all}\;s\in S\}$, respectively.

A subset $Y$ of $X$ is called a \emph{chain}, if any two elements of $Y$ is comparable in $P$. A chain is \emph{maximum} if there is no chain of larger size. The size of a maximum chain in $P$ is called the \emph{height of} $P$, denoted by $height(P)$.

For a given poset $P$, its comparability graph $\comp(P)$ is the graph on $X$ such that $xy\in E(\comp(P))$ if and only if $x\neq y$ and either $x<y$ or $y<x$ in $P$. Note that the comparability graph of a poset does not contain an induced subgraph isomorphic to an odd cycle of length greater than three or the complement of a cycle of length greater than four (see \cite{Trotter}). Therefore, a comparability graph is weakly chordal if it does not contain any induced $C_{2k}$ for $k\geq 3$. We call the poset $P$ a \emph{weakly chordal poset} if $\comp(P)$ is a weakly chordal graph. 

\section{Order-sensitive domination in posets}\label{sect:os-dom-posets}
In this section, we introduce and study the notion of order-sensitive domination in posets. 
\begin{definition}
For a given poset $P$, a dominating set $D$ in $\comp(P)$ is called an \emph{order-sensitive dominating set} in $P$ if either $x\in D$ or else $x\in U_P(D)\cap D_P(D)$ for each $x\in \Mid(P)$. The least size of an order-sensitive dominating set in $P$ is said to be the \emph{order-sensitive domination number} of $P$, denoted by $\gamma_{\os}(P)$. For the brevity, we write \emph{os-dominating set} instead of order-sensitive dominating set throughout. 
\end{definition}

We have the convention that $\gamma_{\os}(P)=\gamma(\comp(P))$, when $\Mid(P)=\varnothing$. In another words, the order-sensitive domination number of $P$ is distinguishable from the domination number of the comparability graph $\comp(P)$ when $\comp(P)$ has a clique of size at least three, which is when $height(P)\geq 3$. Moreover, if $P$ is a poset in which every element is contained in a chain of size at least $3$, then it is obvious that $\gamma_{\os}(P)\leq \min\{|\Mid(P)|,|\Min(P)\cup \Max(P)|\}$.

\begin{remark}
The order-sensitive domination number is not a comparability invariant. Indeed, os-domination numbers of two non-isomorphic posets having isomorphic comparability graphs may differ. For the posets in Figure~\ref{fig:os-dom-notinv}, we have $\gamma_{\os}(P)=2$ and $\gamma_{\os}(R)=1$, while $\comp(P)\cong\comp(R)$.
\end{remark}

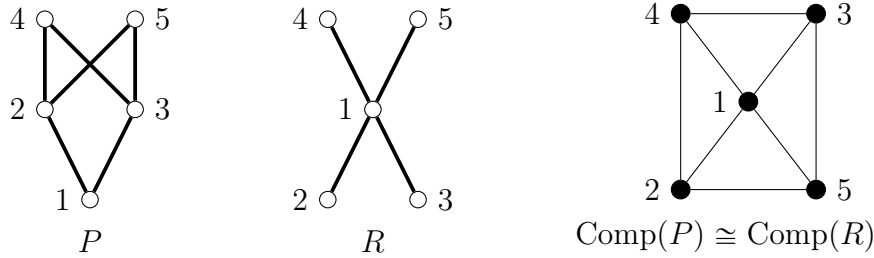
\begin{figure}[ht]
\centering     
\subfigure{
\begin{tikzpicture}[scale=.6]
\node [nodel] at (1,0) (v1) [label=left:$1$]{};
\node [nodel] at (0,2) (v2) [label=left:$2$]{}
               edge [line width=0.5mm] (v1);
\node [nodel] at (2,2) (v3) [label=right:$3$]{}
           edge [line width=0.5mm] (v1);            
\node [nodel] at (0,4) (v4) [label=left:$4$]{}
               edge [line width=0.5mm] (v2)
               edge [line width=0.5mm] (v3);
\node [nodel] at (2,4) (v5) [label=right:$5$]{}
           edge [line width=0.5mm] (v2)
           edge [line width=0.5mm] (v3);
\node at (1,-.2) (v) [label=below:$P$] {};
\end{tikzpicture}
}
\hspace*{8mm}
\subfigure{
\begin{tikzpicture}[scale=.6]
\node [nodel] at (1,2) (v1) [label=left:$1$]{};
\node [nodel] at (0,0) (v2) [label=left:$2$]{}
               edge [line width=0.5mm] (v1);
\node [nodel] at (2,0) (v3) [label=right:$3$]{}
               edge [line width=0.5mm] (v1);
\node [nodel] at (0,4) (v4) [label=left:$4$]{}
               edge [line width=0.5mm] (v1);
\node [nodel] at (2,4) (v5) [label=right:$5$]{}
          	   edge [line width=0.5mm] (v1);
\node at (1,-.2) (v) [label=below:$R$] {};
\end{tikzpicture}
}
\hspace*{8mm}
\subfigure{
\begin{tikzpicture}[scale=.6]
\node [nod3] at (1.5,1.95) (v1) [label=left:$1$]{};
\node [nod3] at (0,0) (v2) [label=left:$2$]{}
               edge [] (v1);
\node [nod3] at (3,0) (v5) [label=right:$5$]{}
               edge [] (v1)
               edge [] (v2);
\node [nod3] at (0,3.9) (v4) [label=left:$4$]{}
               edge [] (v1)
               edge [] (v2);
\node [nod3] at (3,3.9) (v3) [label=right:$3$]{}
          	   edge [] (v1)
               edge [] (v4)
               edge [] (v5);
\node at (1,-.2) (v) [label=below:$\comp(P)\cong \comp(R)$] {};
\end{tikzpicture}
}
\caption{Two non-isomorphic posets with the same comparability graph.}
\label{fig:os-dom-notinv}
\end{figure}

Denote by $\mathcal{P}_l(k)$, the family of posets of height $k$ in which every element is contained in a chain of size at least $l$. Since $l\leq k$ by the definition of $\mathcal{P}_l(k)$, we only consider the case when $l\geq 3$ in the remaining of the paper. In particular, we abbreviate $\PE_3(3)$ to $\PE(3)$.

\begin{proposition}\label{rem:upperb1}
$\gamma_{\os}(P)\leq \ceil*{\frac{n-1}{2}}$ for every $n$-element poset $P\in \mathcal{P}_3(k)$ with $k\geq3$.
\end{proposition}

\begin{proof}
Assume otherwise that $\gamma_{\os}(P)>\ceil*{\frac{n-1}{2}}$ for some poset $P\in \mathcal{P}_3(k)$ with $k\geq3$. Then, we have
$$ \ceil*{\frac{n-1}{2}}<\gamma_{\os}(P)\leq \min\{|\Mid(P)|,|\Min(P)\cup \Max(P)|\}$$
which forces that $\ceil*{\frac{n-1}{2}}+1\leq |\Mid(P)|$ and $\ceil*{\frac{n-1}{2}}+1\leq |\Min(P)\cup \Max(P)|$. Combining these inequalities, we have 

$$n=|\Mid(P)|+|\Min(P)\cup \Max(P)|\geq \ceil*{\frac{n+1}{2}}+\ceil*{\frac{n+1}{2}}\geq n+1 $$
which is a contradiction.
\end{proof}

Next, we will formulate the order-sensitive domination number of a given poset $P\in \mathcal{P}_{3}(k)$, in terms of classical domination number of the comparability graph of another poset constructed from $P$ itself.

Given a poset $P$, and let $\Mid(P)=\{a_1,a_2,\ldots,a_m\}$. We construct the poset $P^*$ as follows: for each element $a_i$ in $\Mid(P)$, we add two new elements $b_i$ and $c_i$ together with covering relations $c_i\prec a_i \prec b_i$. The resulting poset is denoted by $P^*$, in which we set $B=\{b_1,b_2,\ldots,b_m\}$ and $C=\{c_1,c_2,\ldots,c_m\}$ so that $X(P^*)=X(P)\cup B\cup C$ (see Figure~\ref{fig:pstar}). 

\begin{figure}[htb]
\centering     
\subfigure{
\begin{tikzpicture}[scale=1.0]
\node [nodel] at (.3,-1) (a) [label=left:$x$] {};
\node [nodel] at (1.7,-1) (b) [label=right:$y$] {};
\node [nodel] at (-.25,.25) (c) [label=left:$a_1$]{}
	edge [line width=0.5mm] (a)
	edge [line width=0.5mm] (b);
\node [nodel] at (1.5,.25) (d) [label=below:$z$]{};
\node [nodel] at (2.25,.25) (e) [label=right:$a_2$]{}
	edge [line width=0.5mm] (b);
\node [nodel] at (-1,1.3) (f) [label=left:$a_3$]{}
	edge [line width=0.5mm] (c)
	edge [line width=0.5mm] (d);
\node [nodel] at (.8,1.3) (g) [label=right:$a_4$] {}
	edge [line width=0.5mm] (d);
\node [nodel] at (0,2.3) (h) [label=left:$u$] {}
	edge [line width=0.5mm] (f)
	edge [line width=0.5mm] (g);
\node [nodel] at (1.5,2.3) (k) [label=left:$v$] {}
	edge [line width=0.5mm] (g);
\node [nodel] at (3.2,2.3) (m) [label=left:$w$] {}
	edge [line width=0.5mm] (d)
	edge [line width=0.5mm] (e);
\end{tikzpicture} 
}
\hspace*{1cm}
\subfigure{
\begin{tikzpicture}[scale=1.0]
\node [nodel] at (.3,-1) (a) [label=left:$x$] {};
\node [nodel] at (1.7,-1) (b) [label=right:$y$] {};
\node [nodel] at (-.25,.25) (c) [label=left:$a_1$]{}
	edge [line width=0.5mm] (a)
	edge [line width=0.5mm] (b);
\node [nodel] at (-1,-1) (c1) [label=left:$c_1$] {} 
	edge [line width=0.5mm,purple] (c);
\node [nodel] at (1.5,.25) (d) [label=below:$z$]{};
\node [nodel] at (2.25,.25) (e) [label=right:$a_2$]{}
	edge [line width=0.5mm] (b);
\node [nodel] at (3.25,1.3) (b2) [label=right:$b_2$] {} 
	edge [line width=0.5mm,teal] (e);
\node [nodel] at (2.7,-1) (c2) [label=right:$c_2$] {} 
	edge [line width=0.5mm,purple] (e);
\node [nodel] at (-1,1.3) (f) [label=left:$a_3$]{}
	edge [line width=0.5mm] (c)
	edge [line width=0.5mm] (d);
\node [nodel] at (-1.5,2.3) (b3) [label=left:$b_3$] {} 
	edge [line width=0.5mm,teal] (f);
\node [nodel] at (-1.5,.25) (c3) [label=left:$c_3$] {} 
	edge [line width=0.5mm,purple] (f);
\node [nodel] at (.2,1.3) (b1) [label=left:$b_1$] {} 
	edge [line width=0.5mm,teal] (c);
\node [nodel] at (.8,1.3) (g) [label=right:$a_4$] {}
	edge [line width=0.5mm] (d);
\node [nodel] at (2.3,2.3) (k) [label=below:$b_4$] {} 
	edge [line width=0.5mm,teal] (g);
\node [nodel] at (.35,.25) (k) [label=below right:$c_4$] {} 
	edge [line width=0.5mm,purple] (g);
\node [nodel] at (0,2.3) (h) [label=left:$u$] {}
	edge [line width=0.5mm] (f)
	edge [line width=0.5mm] (g);
\node [nodel] at (1.5,2.3) (k) [label=left:$v$] {}
	edge [line width=0.5mm] (g);
\node [nodel] at (3.2,2.3) (m) [label=right:$w$] {}
	edge [line width=0.5mm] (d)
	edge [line width=0.5mm] (e);
\end{tikzpicture} 
}
\caption{$P$ and $P^*$. }
\label{fig:pstar}
\end{figure}

\begin{lemma}\label{lem:os-equal-comp-star}
$\gamma_{\os}(P)=\gamma(\comp(P^*))$ for every poset $P\in \mathcal{P}_3(k)$ with $k\geq 3$.
\end{lemma}

\begin{proof}
We first consider an os-dominating set $D$ in $P$. Note that if $\Mid(P)\subseteq D$, then $D$ is also a dominating set in $\comp(P^*)$ since $\Mid(P)$ dominates all vertices of $B\cup C$ in $\comp(P^*)$.

Thus, we may assume that there exist some $a_i\in \Mid(P)$ such that $a_i\notin D$. By definition of order-sensitive domination, $a_i\in U_P(D)\cap D_P(D)$. This implies that we have two vertices $a',a''\in D\cap X(P)$ such that $a_i<_Pa'$ and $a''<_Pa_i$. It follows that the vertices $a'$ and $a''$ dominate newly added vertices $b_i$ and $c_i$ in $\comp(P^*)$, since $b_i\in B$ (resp. $c_i\in C$) is adjacent to all vertices of $U_p(a_i)$ (resp. $D_P(a_i)$) in $\comp(P^*)$. However, this in turn implies that $D$ is a dominating set in $\comp(P^*)$. Therefore $\gamma(\comp(P^*))\leq |D|$.

Conversely, let $F$ be a dominating set in $\comp(P^*)$. We first note that if $F\cap(B\cup C)= \emptyset$, then $F$ is an os-dominating set for $P$. Indeed, the only way to dominate any vertex $b_i\in B$ (resp. $c_i\in C$) is to take a vertex in $U_P[a_i]$ (resp. $D_P[a_i]$) for $a_i\in \Mid(P)$. Then, either $a_i$ or a vertex in $U_P(a_i)$ (resp. $D_P(a_i)$) belongs to $F$. This coincidences with the definition of order-sensitive domination. Therefore, $F$ is a order-sensitive dominating set in $P$. 

We may therefore suppose that the intersection $F\cap (B\cup C)$ is non-empty. We assume without loss of generality that $b_i\in F\cap (B\cup C)$. It then follows that $F^*=(F-b_i)\cup \{a_i\}$ is still a dominating set in $\comp(P^*)$ containing the middle element $a_i\in F^*$, since, in the graph $\comp(P^*)$, the vertex $a_i$ is adjacent to every vertex in $U_P(a_i)\cap D_P(a_i)\cup \{b_i,c_i\}$. This means that every vertex in $F\cap (B\cup C)$ can be replaced with the corresponding element in $\Mid(P)$ so as to create a new dominating set $F^*\subset X(P)$ in $\comp(P^*)$ with $|F^*|\leq|F|$. However, such a set $F^*$ is obviously an os-dominating set in $P$. Therefore, $\gamma_{\os}(P)\leq |F^*|\leq |F|$.
\end{proof}

We remark that for a given poset $P=(X,\leq)\in \mathcal{P}_3(k)$, the addition of extra comparabilities to $P$ between every pair of maximal and  minimal elements which are incomparable in $P$ does not affect the os-domination number. In other words, denote by $P^{m}$, the poset obtained from $P$ by adding  comparabilities $x<y$ such that $x\parallel_P y$ with $x\in \Min(P)$ and $y\in \Max(P)$. 

\begin{observation}\label{obs:up-down}
$\gamma_{\os}(P)=\gamma_{\os}(P^{m})$ for every poset $P\in \mathcal{P}_3(k)$ with $k\geq 3$.
\end{observation}

Now we show that for every poset $P$ with $height(P)\geq 4$, there exists a height three poset $Q$ on the same ground set such that $\gamma_{\os}(P)=\gamma_{\os}(Q)$. Therefore, the os-domination problem for posets can be reduced to the case of height three. We next describe the construction of such a poset.

\begin{definition}
Let $P=(X,\leq)$ be a poset with $height(P)\geq 4$. The \emph{reduction poset} $\red(P)$ of $P$, is the height three poset on $X$, obtained from $P$ by removing all the comparabilites among the elements of $\Mid(P)$ and preserving the remaining relations in $P$ (see Figure~\ref{fig:reduction}).
\end{definition}

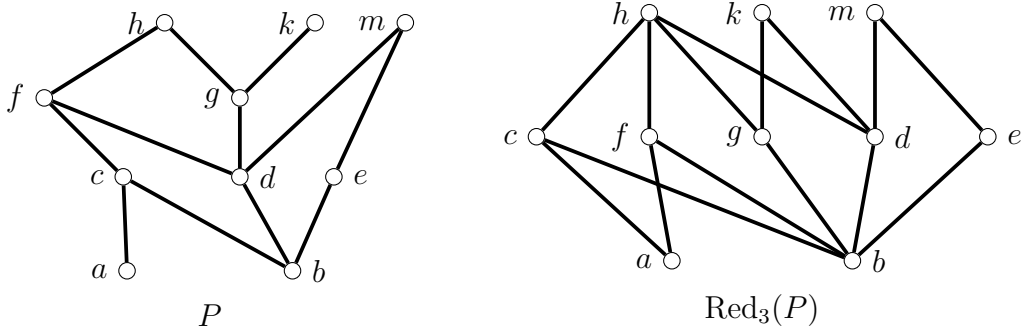
\begin{figure}[htb]
\centering     
\subfigure{
\begin{tikzpicture}[scale=1.0]
\node [nodel] at (-.5,-1) (a) [label=left:$a$] {};
\node [nodel] at (1.7,-1) (b) [label=right:$b$] {};
\node [nodel] at (-.55,.25) (c) [label=left:$c$]{}
	edge [line width=0.5mm] (a)
	edge [line width=0.5mm] (b);
\node [nodel] at (1,.25) (d) [label=right:$d$]{}
	edge [line width=0.5mm] (b);
\node [nodel] at (2.25,.25) (e) [label=right:$e$]{}
	edge [line width=0.5mm] (b);
\node [nodel] at (-1.6,1.3) (f) [label=left:$f$]{}
	edge [line width=0.5mm] (c)
	edge [line width=0.5mm] (d);
\node [nodel] at (1,1.3) (g) [label=left:$g$] {}
	edge [line width=0.5mm] (d);
\node [nodel] at (0,2.3) (h) [label=left:$h$] {}
	edge [line width=0.5mm] (f)
	edge [line width=0.5mm] (g);
\node [nodel] at (2,2.3) (k) [label=left:$k$] {}
	edge [line width=0.5mm] (g);
\node [nodel] at (3.2,2.3) (m) [label=left:$m$] {}
	edge [line width=0.5mm] (d)
	edge [line width=0.5mm] (e);
\node [below] at (0.6,-1.3) {$P$};
\end{tikzpicture} 
}
\hspace*{.5cm}
\subfigure{
\begin{tikzpicture}[scale=1.0]
\node [nodel] at (1.8,-1) (a) [label=left:$a$] {};
\node [nodel] at (4.2,-1) (b) [label=right:$b$] {};
\node [nodel] at (0,0.65) (c) [label=left:$c$]{}
	edge [line width=0.5mm] (a)
	edge [line width=0.5mm] (b);
\node [nodel] at (4.5,0.65) (d) [label=right:$d$]{}
	edge [line width=0.5mm] (b);
\node [nodel] at (6,0.65) (e) [label=right:$e$]{}
	edge [line width=0.5mm] (b);
\node [nodel] at (1.5,0.65) (f) [label=left:$f$]{}
	edge [line width=0.5mm] (a)
	edge [line width=0.5mm] (b);
\node [nodel] at (3,0.65) (g) [label=left:$g$] {}
	edge [line width=0.5mm] (b);
\node [nodel] at (1.5,2.3) (h) [label=left:$h$] {}
	edge [line width=0.5mm] (f)
	edge [line width=0.5mm] (d)
	edge [line width=0.5mm] (g)
	edge [line width=0.5mm] (c);
\node [nodel] at (3,2.3) (k) [label=left:$k$] {}
	edge [line width=0.5mm] (g)
	edge [line width=0.5mm] (d);
\node [nodel] at (4.5,2.3) (m) [label=left:$m$] {}
	edge [line width=0.5mm] (d)
	edge [line width=0.5mm] (e);
\node [below] at (3,-1.3) {$\red(P)$};
\end{tikzpicture}  
}
\caption{A poset $P$ and its reduction $\red(P)$}
\label{fig:reduction}
\end{figure}

\begin{proposition}\label{prop:poset-red3}
If $P\in \mathcal{P}_{3}(k)$ is a poset for some $k\geq 4$, then
$\gamma_{\os}(P)=\gamma_{\os}(\red(P))$.
\end{proposition}
\begin{proof}
Assume that $D\subseteq X(P)$ is an os-dominating set for $P$. We call an element $x\in \Mid(P)$ as a \emph{mid-conflict element} with respect to $D$ if 
\begin{itemize}
\item[•] $x\notin D$, and\\
\item[•] if $z<x<y$ for some $z,y\in D$, then $|\Mid(P)\cap \{z,y\}|\geq 1$.
\end{itemize}

Note that if $x\in \Mid(P)\backslash D$ such that $x$ is not a mid-conflict element with respect to $D$, then there exist 
$y\in D\cap \Max(P)$ and $z\in D\cap \Min(P)$ such that $z<x<y$ in $P$. However, such a domination also holds in $\red(P)$.

Now, assume that $x\in \Mid(P)$ is a mid-conflict element with respect to $D$.  Since $D$ is os-dominating in $P$, it then follows that $a<x<b$ for some $a,b\in D$. Choose $b'\in \Max(P)$ and $a'\in \Min(P)$ such that $b\leq b'$ and $a'\leq a$ hold in $P$, and define $D_x:=(D\setminus \{a,b\})\cup \{a',b'\}$. Observe first that $D_x$ is an os-dominating set in $P$ such that $|D|=|D_x|$. We may clearly repeat this process for every element $u\in \Mid(P)$ such that if $v<u<w$ holds for some $v,w\in D$ and $|\Mid(P)\cap \{v,w\}|\geq 1$. Since $P$ is finite, such a process will eventually terminate creating an os-dominating set $D^*$ in $P$ without any mid-conflict elements. The resulting set $D^*$ is clearly an os-dominating set in $\red(P)$. This verifies that $\gamma_{\os}(\red(P))\leq \gamma_{\os}(P)$.

Conversely, let $T\subseteq X$ be an os-dominating set for
$\red(P)$. We claim that $T$ is itself an os-dominating set for $P$. Let $u\in \Mid(P)$ be given such that $u\notin T$. Then there exist $v\in T\cap \Min(P)$ and $w\in T\cap \Max(P)$ such that $v<u<w$ in $\red(P)$. However, the comparabilities $v<u<w$ also hold in $P$ as well. This proves the claim. Therefore, we conclude that
$\gamma_{\os}(P)\leq \gamma_{\os}(\red(P))$.  
\end{proof}


\section{Order-sensitive domination in graphs}\label{sect:os-dom-graphs}
In this section, we extend the notion of order-sensitive domination to general graphs. We begin with recalling the following bipartite transformation associated to any graph, that was introduced by Alon \cite{NA}.
\begin{definition}
For a graph $G$ with vertex set $V=\{u_1,u_2,\ldots,u_n\}$, the \emph{extended double cover} $B_e(G)$ of $G$ is the bipartite graph on $X\cup Y$, where $X=\{x_1,x_2,\ldots,x_n\}$ and $Y=\{y_1,y_2,\ldots,y_n\}$, in which $x_i$ and $y_j$ are adjacent if and only if $i=j$ or $u_i$ and $u_j$ are adjacent in $G$.
\end{definition}

We next construct a graded poset of height $k\geq 2$ from any given graph as follows.
Let $G$ be a graph on $[n]$ for some positive integer $n$. The ground set of the poset $\mathscr{P}_k(G)$, we associate to $G$ is
$$X(G):=\{x_{(l,i)}\colon l\in [k]\;\textnormal{and}\;i\in [n]\}$$ 
such that for each $r\in [k-1]$ and $i,j\in [n]$,
\begin{itemize}
\item[•] $x_{(r,i)}\prec x_{(r+1,j)}$ in $\mathscr{P}_k(G)$ if and only if either $i=j$ or $ij\in E(G)$ holds.
\end{itemize}

\begin{example} In Figure~\ref{fig:fig2}, we depict the corresponding posets of the $4$-path $P_4$ and the $4$-cycle $C_4$ when $k=3$. 
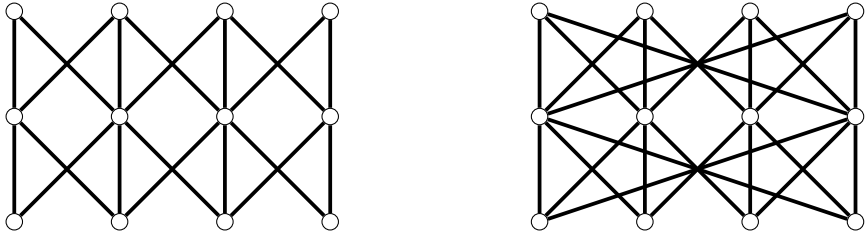
\begin{figure}[ht]
\centering     
\subfigure{
\begin{tikzpicture}[scale=.7]
\node [nodel] at (0,0) (v1) []{};
\node [nodel] at (0,2) (v2) []{}
               edge [line width=0.5mm] (v1);
\node [nodel] at (0,4) (v3) []{}
               edge [line width=0.5mm] (v2);
\node [nodel] at (2,0) (u1) []{}
               edge [line width=0.5mm] (v2);
\node [nodel] at (2,2) (u2) []{}
           edge [line width=0.5mm] (u1)
           edge [line width=0.5mm] (v1)
           edge [line width=0.5mm] (v3);
\node [nodel] at (2,4) (u3) []{}
           edge [line width=0.5mm] (u2)
           edge [line width=0.5mm] (v2);
\node [nodel] at (4,0) (z1) []{}
           edge [line width=0.5mm] (u2);
\node [nodel] at (4,2) (z2) []{}
           edge [line width=0.5mm] (z1)
           edge [line width=0.5mm] (u1)
           edge [line width=0.5mm] (u3);
\node [nodel] at (4,4) (z3) []{}
           edge [line width=0.5mm] (z2)
           edge [line width=0.5mm] (u2);           
\node [nodel] at (6,0) (w1) []{}
           edge [line width=0.5mm] (z2);
\node [nodel] at (6,2) (w2) []{}
           edge [line width=0.5mm] (w1)
           edge [line width=0.5mm] (z1)
           edge [line width=0.5mm] (z3);
\node [nodel] at (6,4) (w3) []{}
           edge [line width=0.5mm] (w2)
           edge [line width=0.5mm] (z2);  
\end{tikzpicture}
}
\hspace*{2cm}
\subfigure{
\begin{tikzpicture}[scale=0.7]
\node [nodel] at (0,0) (v1) []{};
\node [nodel] at (0,2) (v2) []{}
               edge [line width=0.5mm] (v1)
               edge [line width=0.5mm] (w1);
\node [nodel] at (0,4) (v3) []{}
               edge [line width=0.5mm] (v2)
               edge [line width=0.5mm] (w2);
\node [nodel] at (2,0) (u1) []{}
               edge [line width=0.5mm] (v2);
\node [nodel] at (2,2) (u2) []{}
           edge [line width=0.5mm] (u1)
           edge [line width=0.5mm] (v1)
           edge [line width=0.5mm] (v3);
\node [nodel] at (2,4) (u3) []{}
           edge [line width=0.5mm] (u2)
           edge [line width=0.5mm] (v2);
\node [nodel] at (4,0) (z1) []{}
           edge [line width=0.5mm] (u2);
\node [nodel] at (4,2) (z2) []{}
           edge [line width=0.5mm] (z1)
           edge [line width=0.5mm] (u1)
           edge [line width=0.5mm] (u3);
\node [nodel] at (4,4) (z3) []{}
           edge [line width=0.5mm] (z2)
           edge [line width=0.5mm] (u2);           
\node [nodel] at (6,0) (w1) []{}
           edge [line width=0.5mm] (z2);
\node [nodel] at (6,2) (w2) []{}
           edge [line width=0.5mm] (w1)
           edge [line width=0.5mm] (z1)
           edge [line width=0.5mm] (z3)
           edge [line width=0.5mm] (v1);
\node [nodel] at (6,4) (w3) []{}
           edge [line width=0.5mm] (w2)
           edge [line width=0.5mm] (z2)
           edge [line width=0.5mm] (v2);  
\end{tikzpicture}
}
\caption{$\mathscr{P}_3(P_4)$ and $\mathscr{P}_3(C_4)$.}
\label{fig:fig2}
\end{figure}
\end{example}

Notice that for given a graph $G$, the comparability graph of the poset $\mathscr{P}_2(G)$ corresponds to the bipartite graph $B_e(G)$.

\begin{lemma}\label{lem:os-dom-3-graphs}
$\gamma_{\os}(\mathscr{P}_3(G))=\min\{|A|+2\gamma(G-A): \, A\subseteq V\}$ for every graph $G$.
\end{lemma}

\begin{proof}
We write $\theta(G)=\min\{|A|+2\gamma(G-A): \, A\subseteq V\}$. Suppose that $D$ is a minimum os-dominating set for $\mathscr{P}_3(G)$. If we define $A:=\{j\in [n]: \, x_{(2,j)}\in D\}$, $B_1:=\{l\in [n]: \, x_{(1,l)}\in D \}$ and $B_3:=\{ q\in [n]: \, x_{(3,q)}\in D \}$, we claim that each of the sets $B_1$ and $B_3$ is a dominating set for $G-A$. Assume that $p \in [n]\setminus A$. Since $x_{(2,p)} \notin D$ and $D$ is an os-dominating set, there exist $x_{(1,i)}, x_{(3,j)}\in D$ such that $x_{(1,i)}<x_{(2,p)}<x_{(3,j)}$ holds in $\mathscr{P}_3(G)$. However, this means that $i\in B_1$, $j\in B_3$ and $ip, jp\in E$. It follows that $B_1$ and $B_3$ are dominating sets for $G-A$ as claimed. Furthermore, we may partition the set $D$ as $D=D_1 \cup D_2 \cup D_3$, where $D_s=D\cap \{x_{(s,t)}: \, t\in [n] \}$ for $s\in [3]$. Therefore we conclude that
$$\theta(G)\leq |A|+|B_1|+|B_3|\leq |D_2|+(|D_1|+|D_3|)=\gamma_{\os}(\mathscr{P}_3(D)).$$
Conversely, assume that $\theta(G)=|A|+2\gamma(G-A)$ for some subset $A\subseteq [n]$. Let $B\subseteq [n]\setminus A$ be a minimum dominating set for the graph $G-A$. We define a set
$$C:=\{x_{(1,i)}, x_{(2,j)}, x_{(3,i)}: \, i\in B \text{ and }j\in A \}\subseteq X(G)$$
and claim that $C$ is an os-dominating set for $\mathscr{P}_3(G)$. Let $x_{(p,q)}\in X(G) \setminus C$ be given. Assume first that $q\in A$. If $p\in \{1,3\}$, then the vertex $x_{(2,q)}\in C$ dominates both $x_{(1,q)}$ and $x_{(3,q)}$. Secondly, if $q\notin A$, there exists $h\in B$ such that $qh\in E$, since $B$ is a dominating set of $G-A$. It then follows that $x_{(1,h)},x_{(3,h)}\in C$ and $x_{(1,h)}<x_{(2,q)}<x_{(3,h)}$ in $\mathscr{P}_3(G)$. Finally, we conclude that
$$\gamma_{\os}(\mathscr{P}_3(G))\leq |C|=|A|+2|B|=\theta(G)$$
as claimed.
\end{proof}

\begin{theorem}\label{thm:os-dom-roman}
$\gamma_{\os}(\mathscr{P}_3(G))=\gamma_{\rom}(G)$ for every connected graph $G$ with order at least two.
\end{theorem}
\begin{proof}
The inequality $\gamma_{\os}(\mathscr{P}_3(G))\leq \gamma_{\rom}(G)$ simply follows from Lemma~\ref{lem:os-dom-3-graphs} together with the fact that $\gamma_{\rom}(G)=\min\{|S|+2\gamma(G-S): S\;\textnormal{is a}\;2\textnormal{-packing}\}$ by~\citep[Corollary $1$]{CDHH}.

For the converse, we follow the proof of Lemma~\ref{lem:os-dom-3-graphs}. Let $D$ be an os-dominating set for $\mathscr{P}_3(G)$, and let $A$, $B_1$ and $B_3$ be the subsets of $V(G)=[n]$ defined as in Lemma~\ref{lem:os-dom-3-graphs}. Assume without loss of generality that $|B_1|\leq |B_3|$. We then define a function $f_D\colon [n]\to \{0,1,2\}$ by
\begin{equation*}
f_D(i):=\begin{cases}
1,& \textnormal{if}\;i\in A,\\
2,& \textnormal{if}\;i\in B_1,\\
0,& \textnormal{otherwise},\\ 
\end{cases}
\end{equation*} 
and claim that $f_D$ is a Roman dominating function. Suppose that $j\in [n]$ is a vertex with $f_D(j)=0$. This means that $x_{(2,j)}\notin D$. So, there exists $t\in B_1$ so that $x_{(1,t)}<x_{(2,j)}$ in $\mathscr{P}_3(G)$. It then follows that $f_D(t)=2$ and $jt\in E$. This proves the claim. Now,
we conclude that
$$\we(f_D)=|A|+2|B_1|\leq |D|=\gamma_{\os}(\mathscr{P}_3(G)).$$
\end{proof}

\begin{example} For each $n\geq 3$, we have
$$\gamma_{\os}(\mathscr{P}_3(P_n))=\gamma_{\os}(\mathscr{P}_3(C_n))=\gamma_{\rom}(P_n)=\gamma_{\rom}(C_n)=\ceil*{\dfrac{2n}{3}}.$$
\end{example}

Taking Theorem~\ref{thm:os-dom-roman} into account, the following inequalities are known (see~\citep[Theorem $2.1$]{Hedetniemi-Walsh}). However, we choose to include its proof. 

\begin{theorem}\label{thm:3-os-dom-graphs-bounds}
$\gamma_t(G)\leq \gamma_{\os}(\mathscr{P}_3(G))\leq 2\gamma(G)$ for every graph $G$ without isolated vertices.
\end{theorem}

\begin{proof}
Let $D$ be a minimum os-dominating set for $\mathscr{P}_3(G)$. Then there exists a subset $A\subset V$ such that $|D|=|A|+2\gamma(G-A)$ . Let $S$ be a minimum dominating set for $G-A$ such that $|D|=|A|+2|S|$. We construct a dominating set $S'$ by adding to $S$, exactly one neighbor of each vertices of $S$. Clearly, we then have $|S'|\leq 2|S|$.

Let $A'$ be a set of isolated vertices of $G[A]$ such that $N(A')\cap S'=\varnothing$. Let $N(A')\cap (G-A)=T $. By the choice of $A'$, we have $T\cap S'=\varnothing$. We take a minimum subset of $T$, say $T'$, dominating $A'$. We obviously have $|T'|\leq |A'|$.
Since $T\cap S'=\varnothing$ and $S'$ is a dominating set for $G-A$, each vertex of $T'$ has a neighbor in $S'$. Therefore we obtain a total dominating set 
$$M=S'\cup T' \cup (A-A')$$
for $G$. It is clear that $|M|\leq |D|$. Thus we have
$$\gamma_t(G)\leq |M|\leq |D|=\gamma_{\os}(\mathscr{P}_3(G)).$$
The last inequality simply follows from Lemma~\ref{lem:os-dom-3-graphs} by taking $A=\emptyset$.
\end{proof}

Our final aim in this section is to consider the case $k=4$ for which we prove that the order-sensitive domination number can be directly computed from the domination number of the underlying graphs.

\begin{theorem}\label{thm:4-gammaos-gamma}
$\gamma_{\os}(\mathscr{P}_4(G))=2\gamma(G)$ for every graph $G$.
\end{theorem}

In order to prove Theorem~\ref{thm:4-gammaos-gamma}, we initially introduce an operation that constructs a graded poset of height four from any given bipartite graph without any isolated vertex.

\begin{figure}[ht]
\centering     
\subfigure{
\begin{tikzpicture}[scale=0.7]
\node [nod3] at (0,2) (u1) [label=below:$a$]{};  
\node [nod3] at (2,2) (u2) [label=below:$b$]{};  
\node [nod3] at (4,2) (u3) [label=below:$c$]{};  
\node [nod3] at (6,2) (u4) [label=below:$d$]{};  
\node [nod3] at (1,4) (u5) [label=above:$u$]{}  
          edge [] (u1)
          edge [] (u2);
\node [nod3] at (3,4) (u6) [label=above:$v$] {}  
           edge [] (u1)
           edge [] (u2)
           edge [] (u3);
\node [nod3] at (5,4) (u7) [label=above:$w$] {}  
           edge [] (u2)
           edge [] (u3)
           edge [] (u4);           
\node at (3,-.5) (u) [label=below:$B$] {};
\end{tikzpicture}   
}
\hspace*{1.5cm}
\subfigure{
\begin{tikzpicture}[scale=.7]
\node [nodel] at (0,0) (a1) [label=left:$a^1$]{};  
\node [nodel] at (2,0) (b1) [label=left:$b^1$]{};  
\node [nodel] at (4,0) (c1) [label=left:$c^1$]{};  
\node [nodel] at (6,0) (d1) [label=left:$d^1$]{};  
\node [nodel] at (0,4) (a3) [label=left:$a^3$]{};  
\node [nodel] at (2,4) (b3) [label=left:$b^3$]{};  
\node [nodel] at (4,4) (c3) [label=left:$c^3$]{};  
\node [nodel] at (6,4) (d3) [label=left:$d^3$]{};  
\node [nodel] at (1,2) (u2) [label=right:$u^2$]{}  
          edge [line width=0.5mm] (a1)
          edge [line width=0.5mm] (b1)
          edge [line width=0.5mm] (a3)
          edge [line width=0.5mm] (b3);
\node [nodel] at (3,2) (v2) [label=right:$v^2$] {}  
           edge [line width=0.5mm] (a1)
           edge [line width=0.5mm] (b1)
           edge [line width=0.5mm] (c1)
           edge [line width=0.5mm] (a3)
           edge [line width=0.5mm] (b3)
           edge [line width=0.5mm] (c3);
\node [nodel] at (5,2) (w2) [label=right:$w^2$] {}  
           edge [line width=0.5mm] (b1)
           edge [line width=0.5mm] (c1)
           edge [line width=0.5mm] (d1)      
           edge [line width=0.5mm] (b3)
           edge [line width=0.5mm] (c3)
           edge [line width=0.5mm] (d3);  
\node [nodel] at (1,6) (u4) [label=right:$u^4$]{}  
          edge [line width=0.5mm] (a3)
          edge [line width=0.5mm] (b3);
\node [nodel] at (3,6) (v4) [label=right:$v^4$] {}  
           edge [line width=0.5mm] (a3)
           edge [line width=0.5mm] (b3)
           edge [line width=0.5mm] (c3);
\node [nodel] at (5,6) (w4) [label=right:$w^4$] {}  
           edge [line width=0.5mm] (b3)
           edge [line width=0.5mm] (c3)
           edge [line width=0.5mm] (d3);  
 \node at (6.5,0) (u) [label=right:$L_1$] {};
 \node at (6.5,2) (u) [label=right:$L_2$] {};
 \node at (6.5,4) (u) [label=right:$L_3$] {};
 \node at (6.5,6) (u) [label=right:$L_4$] {};
\node at (3,-.5) (u) [label=below:$B_4$] {};
\end{tikzpicture}
}
\caption{A bipartite graph $B$ and the poset $B_4$.}
\label{fig:B4}
\end{figure}
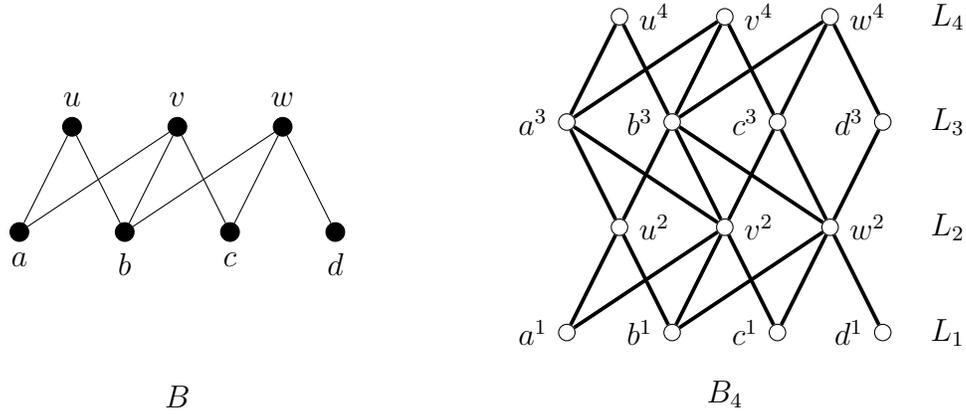

We may consider every bipartite graph $B=(X,Y;E)$ without any isolated vertex as a height two poset in which $\Max(B)=Y$ and $\Min(B)=X$. We then construct a graded height four poset $B_4$ as follows. If we 
denote the $i^{\textnormal{th}}$-layer of $B_4$ by $L_i$, then
$L_4:=\Max(B_4)=\{y^4\colon y\in Y\}$, $L_3:=\{x^3\colon x\in X\}$,
$L_2:=\{y^2\colon y\in Y\}$ and $L_1=\Min(B_4)=\{x^1\colon x\in X\}$, where covering relations are all inherited from $B$ itself. In other words, $x^i\prec y^{i+1}$ and $y^j\prec x^{j+1}$ if and only if $xy\in E$ for every pair $i\in \{1,3\}$ and $j=2$.

\begin{proposition}\label{prop:b4-gamma-total}
If $B$ is a bipartite graph, then $\gamma_{\os}(B_4)=\gamma_t(B)$.
\end{proposition}

\begin{proof}
If $D$ is a total dominating set for $B$, then the set
$$D_4=\{y^4\colon \; y\in D\cap Y\}\cup \{x^1\colon \; x\in D\cap X\}$$
is clearly an os-dominating set for $B_4$.

For the converse, let $D$ be a minimum os-dominating set for $B_4$. We define $S_1\subset L_4$ (resp., $S_2\subset L_1)$ as the set of vertices containing exactly one neighbor of each vertex of $D\cap L_3$ (resp. $D\cap L_2$). Then the set
$$D'=(D\backslash(L_2\cup L_3))\cup (S_1\cup S_2)$$
is an os-dominating set for $B_4$, since each vertex of $L_3$ (resp. $L_2$) has a neighbor in $L_4$ (resp. $L_1$). Clearly we have $|D'|\leq |D|$. Now let $D''$ be the subset of $V(B)$ corresponding to the vertices in $D'$ in $B_4$. Clearly, $D''$ is a total dominating set for $B$, since $D''\cap Y$ (resp. $D''\cap  X$) dominates all the vertices of $X$ (resp. $Y$). Therefore we have
$$\gamma_t(B)\leq |D''|\leq |D'| \leq |D|=\gamma_{\os}(B_4).$$
\end{proof}

\begin{proposition}\label{prop:t-extdoub}
$\gamma_t(B_e(G))=2\gamma(G)$ for every graph $G$.
\end{proposition}
\begin{proof}
Let $D$ be a dominating set for $G$, where $V=V(G)=\{v_1,\ldots,v_n\}$.
If we define
\begin{equation*}
D^*:=\{x_i\colon v_i\in D\}\cup \{y_j\colon v_j\in D\},
\end{equation*}
then the set $D^*$ is a total dominating set for $B_e(G)$ of size $2|D|$.

Assume next that $F\subseteq X\cup Y$ is a total dominating set for $B_e(G)$. Observe that each of the sets
\begin{equation*}
F_1:=\{v_i\colon x_i\in F\cap X\} \quad \textnormal{and}\quad F_2:=\{v_j\colon y_j\in F\cap Y\}
\end{equation*}
are dominating sets for $G$. Hence, we have
\begin{equation*}
2\gamma(G)\leq |F_1|+|F_2|=|F|=\gamma_t(B_e(G)).
\end{equation*}

\end{proof}

\begin{proof}[{\bf Proof of Theorem~\ref{thm:4-gammaos-gamma}}]
Since $(B_e(G))_4\cong \mathscr{P}_4(G)$,  the claim follows from Propositions~\ref{prop:b4-gamma-total} and~\ref{prop:t-extdoub}.
\end{proof}

\section{Middle graphs and Posets with Helly property}\label{sec: middle}

In this section, we introduce middle graphs of posets and prove that under some restrictions, the order-sensitive domination numbers can be detected from these graphs.
\begin{definition}
Let $P\in \mathcal{P}_3(k)$ be a poset for some $k\geq 3$. The graph $\midd(P)$ defined on the set $\Mid(P)$ by $ab\in E(\midd(P))$ if and only if $CU_P(\{a,b\})\neq \emptyset$ and $CD_P(\{a,b\})\neq \emptyset$ is called the \emph{middle graph} of $P$.
\end{definition}

\begin{observation}\label{obser:red-mid}
If $P$ is a poset in $\mathcal{P}_3(k)$ with $k\geq 4$, then $\midd(P)$ is isomorphic to $\midd(\red(P))$.
\end{observation}

Let $\FE$ be a family of subsets of a ground set $X$. A subfamily $\FE'\subseteq \FE$ is called \emph{intersecting} if the intersection of every pair of sets in $\FE'$ is non-empty. The family $\FE$ is said to have the \emph{Helly property} if $\bigcap \FE'\neq \emptyset$ for every intersecting subfamily $\FE'\subseteq \FE$.

Now, let $P=(X,\leq)$ be a poset. We consider two families $\UE_P:=\{U_P(x)\colon x\in \Mid(P)\}$ and $\DE_P:=\{D_P(x)\colon x\in \Mid(P)\}$.

\begin{definition}
A poset $P\in \PE_3(k)$  for some $k\geq 3$ is called a \emph{Helly poset} if both families $\UE_P$ and $\DE_P$ have the Helly property.
\end{definition}

In order to justify our generalization, consider a family $\FE$ of non-empty subsets of a ground set such that $\bigcup_{F\in \FE}F=X$. We duplicate the ground set as $X':=\{x'\colon x\in X\}$, $X'':=\{x''\colon x\in X\}$, and then define a graded height three poset $P_3(\FE)$ by $\Max(P_3(\FE))=X''$, $\Mid(P_3(\FE))=\FE$ and $\Min(P_3(\FE))=X'$ such that $x'<F<y''$ if and only if $x,y\in F$.

\begin{observation}
The family $\FE$ has the Helly property if and only if $P_3(\FE)$ is a Helly poset.
\end{observation}

For a given poset $P\in \PE_3(k)$ with $k\geq3$, we define two graphs by
\begin{equation*}
H_{u}(P):=\comp(P)[\Mid(P)\cup \Max(P)]\quad \textnormal{and}\quad H_{d}(P):=\comp(P)[\Min(P)\cup \Mid(P)].
\end{equation*}

\begin{proposition}
A poset $P\in \PE_3(k)$ with $k\geq3$ is a Helly poset if the graphs $H_{u}(P)$ and $H_{d}(P)$ are $C_6$-free.
\end{proposition}
\begin{proof}
We only prove the claim for $H_u=H_u(P)$, and note that a similar argument applies to the graph $H_d$. Let $\UE_P(S)=\{U_P(x)\colon x\in S\}$ be a pair-wise intersecting family for some $S\subseteq \Mid(P)$ (with $|S|\geq 3$) such that $CU_P(S)= \emptyset$, and let $S$ be minimal with this property. Observe that the set $S$ must be an antichain. In fact, if $x<y$ in $P$ for some $x,y\in S$, then $CU_P(S\setminus \{x\})= \emptyset$, which is not possible by the minimality of $S$. Let $T_S\subseteq \Mid(P)\cup \Max(P)$ be a minimal subset such that for every pair $x,y\in S$, there exists $q\in T_S$ satisfying $x,y<q$ in $P$. We claim that $T_S$ is also an antichain. Indeed, if $a<b$ for some $a,b\in T_S$, then for every pair $x,y\in S$, there exists $q\in T_S\setminus \{a\}$ satisfying $x,y<q$ in $P$, which contradicts the minimality of $T_S$. The minimality of $T_S$ further implies that for each $p\in T_S$, there exists a unique pair $x_p,y_p\in S$ such that $x_p,y_p<p$ holds. Observe that if $p,q\in T_S$ with $p\neq q$,
then $|\{x_p,y_p,x_q,y_q\}|\geq 3$. So, consider a $3$-element set $\{x_p,y_p,x_q\}$. If $x_p<q$, then $y_p\nless q$ and there exists $z\in T_S\setminus \{p,q\}$ such that $y_p,x_q<z$ in $P$. However, the minimality of $T_S$ forces that $x_p\nless z$. It then follows that $\{x_p,p,y_p,z,x_q,q\}$ induces a $C_6$ in $H_u$, a contradiction. 
We may therefore assume that $x_p\nless q$ in $P$. A similar reasoning implies that $y_p\nless q$. However, since $\UE(S)$ is pair-wise intersecting, there exist $z_1,z_2\in T_S\setminus \{p,q\}$ such that
$x_p,x_q<z_1$ and $y_p,x_q<z_2$. In such a case, we then conclude that $\{x_p,z_1,x_q,z_2,y_p,p\}$ induces a $C_6$ in $H_u$, a contradiction.
\end{proof}

\begin{corollary}\label{cor:wcp-ekr}
If $\comp(P)$ is a $C_6$-free graph for some poset $P\in \PE_3(k)$ with $k\geq3$, then $P$ is a Helly poset.
\end{corollary}

\begin{proposition}
If $P\in \PE_3(k)$ is a Helly poset for some $k\geq 3$, then $$\gamma_{\os}(P)\leq 2\chi(\overline{\midd(P)}).$$
\end{proposition}
\begin{proof}
If $C_1,\ldots,C_r$ are cliques in $\midd(P)$ corresponding a coloring of its complement, then the set $\{a_1,\ldots,a_r\}\cup \{b_1,\ldots,b_r\}$ is an os-dominating set for $P$, where $a_i\in D_P(C_i)$ and $b_i\in U_P(C_i)$ for each $1\leq i\leq r$. 
\end{proof}

\begin{definition}
Let $P\in \PE_3(k)$ be a Helly poset. For a given subset $S\subseteq \Mid(P)$, we set $\UE_P(S)=\{U_P(x)\colon x\in S\}$ and $\DE_P(S)=\{D_P(x)\colon x\in S\}$. Then, $P$ is called a \emph{complete Helly poset} provided that $\UE_P(S)$ is intersecting if and only if $\DE_P(S)$ is intersecting for every subset $S\subseteq \Mid(P)$.
\end{definition}

\begin{lemma}\label{lem:red-ekr}
Let $P\in \PE_3(k)$ be a poset for some $k\geq 4$. If $P$ is a Helly poset, then so is $\red(P)\in \PE(3)$. In particular, if $P$ is a complete Helly poset, then so is $\red(P)$.
\end{lemma}

\begin{proof}
We first set $R:=\red(P)$, and note that it is sufficient to show that $\UE_{R}=\{U_R(x)\colon x\in \Mid(R)\}$ has the  Helly property, since the case for $\DE_{R}$ can be treated similarly. Suppose that $\{U_{R}(x)\colon x\in S\}\subseteq \UE_{R}$ is an intersecting family for some $S\subseteq  \Mid(R)$ with $|S|\geq 3$. Then, it follows that $U_{R}(x)\cap U_{R}(y)\neq \emptyset$ for every pair of elements $x,y\in S$. However, this forces that the family $\{U_{P}(x)\colon x\in S\}\subseteq \UE_{P}$ is intersecting, since $\emptyset\neq U_{R}(x)\cap U_{R}(y) \subseteq U_{P}(x)\cap U_{P}(y)$ for $x,y\in S$. The  Helly property of $\UE_{P}$ implies that there exists an element $w \in CU_P(S)$. If $w\in \Max(P)$, then $w \in CU_{R}(S)$.  If $w\notin \Max(P)$, then $z \in CU_{R}(S)$ for each $z\in \Max(P)$ with $w<z$. 

For the second claim, let $\{U_{R}(x)\colon x\in S\}\subseteq \UE_{R}$ be an intersecting family for some $S\subseteq  \Mid(R)$. Once again, $\{U_{P}(x)\colon x\in S\}\subseteq \UE_{P}$ is intersecting by the above argument. Completeness of $P$ implies that $\{D_{P}(x)\colon x\in S\}\subseteq \DE_{P}$ is intersecting. It then follows that $D_{R}(x)\cap D_{R}(y)=(D_{P}(x)\cap D_{P}(y))\cap \Min(P)\neq \emptyset$ for $x,y\in S$, since $D_R(a)=D_P(a)\cap \Min(P)$ for each $a\in \Mid(P)$. A similar argument applies when the family $\{D_{R}(x)\colon x\in S\}\subseteq \DE_{R}$ is intersecting. This completes the proof.
\end{proof}

We note that even if a Helly poset $P$ is self-dual, it does not need to be necessarily complete (see Figure~\ref{fig:self-dual}). 
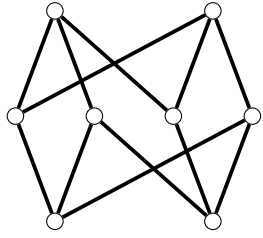
\begin{figure}[ht]
\centering     
\begin{tikzpicture}[scale=.7]
\node [nodel] at (0,2) (1) []{};  
\node [nodel] at (1.5,2) (2) []{};  
\node [nodel] at (3,2) (3) []{};  
\node [nodel] at (4.5,2) (4) []{};  
\node [nodel] at (.75,0) (x) [] {} 
           edge [line width=0.5mm] (1)
           edge [line width=0.5mm] (2)
           edge [line width=0.5mm] (4); 
\node [nodel] at (3.75,0) (y) [] {}  
           edge [line width=0.5mm] (2)
           edge [line width=0.5mm] (3)
           edge [line width=0.5mm] (4);
\node [nodel] at (.75,4) (a) [] {}
           edge [line width=0.5mm] (1)
           edge [line width=0.5mm] (2)
           edge [line width=0.5mm] (3);
\node [nodel] at (3.75,4) (b) [] {} 
           edge [line width=0.5mm] (1)
           edge [line width=0.5mm] (3)
           edge [line width=0.5mm] (4);
\end{tikzpicture}
\caption{A non-complete self-dual  Helly poset.}
\label{fig:self-dual}
\end{figure}

Whenever $P\in \PE_3(k)$ is a complete  Helly poset, the calculation of order-sensitive domination number of $P$ can be interpreted as a weighted clique partition number of the associated middle graph $\midd(P)$.  In more detail, for a given graph $G$, we define the \emph{weight} of a (non-empty) clique $C$ in $G$ by 

\begin{equation*}
\we(C):=\begin{cases}
1,& \quad \textnormal{if}\;|C|=1,\\
2,& \quad \textnormal{if}\;|C|\geq 2.
\end{cases}
\end{equation*}
Furthermore, if $\CE=\{C_1,\ldots,C_n\}$ is a clique partition of $G$, that is, the pairwise intersection of cliques in $\CE$ is empty and $\bigcup_{i=1}^n C_i=V(G)$, then its weight is defined to be the integer 
\begin{equation*}
\we(\CE)=\sum_{i=1}^n \we(C_i).
\end{equation*}

\begin{definition}
We define the \emph{weighted clique partition number} of a graph $G$  by $\we(G):=\min \{\we(\CE)\colon \CE\;\textnormal{is a clique partition of}\;G\}$.
\end{definition}

\begin{theorem}\label{thm:comekr-wclique}
If $P\in \PE_3(k)$ with $k\geq 3$ is a Helly poset, then
$\gamma_{\os}(P)\leq \we(\midd(P))$. In particular, $\gamma_{\os}(P)= \we(\midd(P))$ if $P$ is a complete Helly poset.
\end{theorem}
\begin{proof}
By taking Observation~\ref{obser:red-mid} and Lemma~\ref{lem:red-ekr} into account, it suffices to prove the claim when $k=3$.

Assume that $\CE$ is a clique partition of $\midd(P)$.
Decompose $\CE$ as $\CE_1\cup \CE_2$, where $\CE_1$ consists of all cliques in $\CE$ of order at least three. For $C\in \CE_1$, we choose
$a_C\in D_P(C)$ and $b_C\in U_P(C)$, and form the set
\begin{equation*}
D(\CE):=\bigcup_{C\in \CE_1} \{a_C,b_C\}\cup \bigcup_{C\in \CE_2} C.
\end{equation*}
Clearly, $D(\CE)$ is an os-dominating set for $P$ of size $\we(\CE)$. Thus, we have $\gamma_{\os}(P)\leq \we(\CE)$.

Suppose next that $P$ is a complete Helly poset, and let $D\subseteq X$ be an os-dominating set for $P$ of minimum size. We partition $D$ as $D_{\max}\cup D_{\md}\cup D_{\min}$ in the obvious way, and claim first that $|D_{\max}|=|D_{\min}|$. Write
$D_{\max}=\{b_1,\ldots,b_k\}$ and $D_{\min}=\{a_1,\ldots,a_l\}$, and assume without loss of generality that $l<k$. For each $1\leq i\leq l$,
we set $S_i:=\{x\in \Mid(P)\colon a_i<x\}$. Since $\DE_P(S_i)$ is intersecting, then so is $\UE_P(S_i)$ by the completeness of $P$. So,
if we choose $b'_i\in CU_P(S_i)$, we claim that
$$D':=\{b'_1,\ldots,b'_l\}\cup D_{\md}\cup D_{\min}$$ 
is an os-dominating set for $P$. Indeed, let $x\in \Mid(P)\setminus D_{\md}$ be given. Since $D$ is an os-dominating set, there exist $i\in [l]$ and $j\in [k]$ such that $a_i<x<b_j$. However, we then have $x<b'_i$, since $x\in S_i$. This yields the desired contradiction, since $|D'|<|D|$.

Finally, the family $\SE=\{S_1,\ldots,S_l\}\cup \{\{x\}\colon x\in D_{\md}\}$ provides a clique partition of $\midd(P)$ such that $\we(\SE)=|D|$. This completes the proof.
\end{proof}

\begin{remark}
The (complete) Helly property in Theorem~\ref{thm:comekr-wclique} is essential. For example, if we consider $G=C_4$, we would have $\midd(\mathscr{P}_3(C_4))\cong C_4^2\cong K_4$ so that $\we(\midd(\mathscr{P}_3(C_4)))=2$, while
$\gamma_{\os}(\mathscr{P}_3(C_4))=3$. Furthermore, the poset $P$ depicted in Figure~\ref{fig:non-complete} is a Helly poset, which is not complete. In particular, we note that $\gamma_{\os}(P)=3<4=\we(\midd(P))$.
\end{remark}

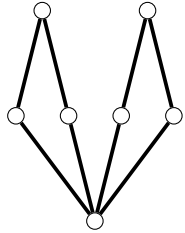
\begin{figure}[ht]
\centering     
\begin{tikzpicture}[scale=.7]
\node [nodel] at (1,2) (1) []{};  
\node [nodel] at (2,2) (2) []{};  
\node [nodel] at (3,2) (3) []{};  
\node [nodel] at (4,2) (4) []{};  
\node [nodel] at (2.5,0) (x) [] {}  
           edge [line width=0.5mm] (1)
           edge [line width=0.5mm] (2)
           edge [line width=0.5mm] (3)
           edge [line width=0.5mm] (4);
\node [nodel] at (1.5,4) (a) [] {}
           edge [line width=0.5mm] (1)
           edge [line width=0.5mm] (2);
\node [nodel] at (3.5,4) (b) [] {} 
           edge [line width=0.5mm] (3)
           edge [line width=0.5mm] (4);
\end{tikzpicture}
\caption{A non-complete  Helly poset $P$ with $\gamma_{\os}(P)=3<4=\we(\midd(P))$.}
\label{fig:non-complete}
\end{figure}

We denote by $\vdt(G)$, the maximum number of vertex disjoint triangles in a given graph $G$ (see \cite{ZLJ}). Note that the problem of partitioning a graph into vertex disjoint triangles corresponds to the well-known $3$-dimensional matching problem, which also appears in Karp's paper \cite{Karp1972} on the theory of computational complexity.

\begin{corollary}
If $P\in \PE(3)$ is a complete  Helly poset, then $$\gamma_{\os}(P)\leq |\Mid(P)|-\vdt(\midd(P)),$$ 
with equality if the clique number of $\midd(P)$ is three.
\end{corollary}

We next consider the poset $\mathscr{P}_3(G)$, and characterize graphs for which their associated posets are complete  Helly posets. Recall that the $3$-sun is the graph obtained from $C_6$ by turning an independent set of size three into a clique.

\begin{lemma}\label{lem:bipext-double-ekr}
$B_e(G)$ is $C_6$-free if and only if $G$ is $(3\text{-sun},C_4,C_5,C_6)$-free.
\end{lemma}
\begin{proof}
Since any induced cycle of length $4$, $5$ or $6$ as well as an induced $3\text{-sun}$ yields an induced $C_6$ in $B_e(G)$, we only verify the sufficiency .

We write $V=\{v_1,\ldots,v_n\}$, and assume that $B_e(G)$ contains an induced $C_6$ on $$\{x_{i_1},y_{i_2},x_{i_3}, y_{i_4}, x_{i_5}, y_{i_6}\}.$$

Suppose that there exists $r\in [6]$ such that $i_r=i_{r+1}$, where indices are taken modulo $6$. Assume without loss of generality that $i_1=i_2$. Since $x_{i_1}y_{i_4}\notin E(B_e(G))$, we have $i_3\neq i_4$. A similar reasoning implies that $i_5\neq i_6$. It then follows that either $i_4=i_5$ or else $i_4\neq i_5$. In the former case, $G$ contains an induced $C_4$, while in the latter case, it contains an induced $C_5$. 

Now we are left with the case that $i_r\neq i_s$ for all distinct $r$ and $s$. In such a case, the subgraph induced by the set $K=\{v_{i_1},v_{i_2},v_{i_3}, v_{i_4}, v_{i_5}, v_{i_6}\}$ contains a $C_6$ with the edges in the cyclic fashion together with possible edges among the vertices in $S=\{v_{i_1},v_{i_3}, v_{i_5}\}$ or among the vertices in $T=\{v_{i_2},v_{i_4}, v_{i_6}\}$. If one of $S$ and $T$ is an independent set while the other is a clique in $G$, then $K$ induces a $3\text{-sun}$ in $G$. If both are independent, then $G[K]\cong C_6$. Observe that in every remaining case, either $C_4$ or $C_5$ appears as an induced subgraph of $G[K]$. This completes the proof.
\end{proof}

\begin{corollary}\label{cor:h3-cekr}
$\mathscr{P}_3(G)$ is a complete  Helly poset if $G$ is $(3\text{-sun},C_4,C_5,C_6)$-free.
\end{corollary}

\begin{corollary}\label{cor:ekr-graph-bound}
If $G$ is $(3\text{-sun},C_4,C_5,C_6)$-free, then 
$$\gamma_{\os}(\mathscr{P}_3(G))=\gamma_{\rom}(G)=\we(G^2)\leq 2\chi(\overline{G^2}).$$
\end{corollary}
\begin{proof}
We note that $\mathscr{P}_3(G)$ is a complete  Helly poset by Corollary~\ref{cor:h3-cekr}. Therefore, the claim follows from Theorems~\ref{thm:os-dom-roman} and \ref{thm:comekr-wclique}, since $\midd(\mathscr{P}_3(G))\cong G^2$.
\end{proof}
Notice that the class of strongly chordal graphs is not a subclass of  $(3\text{-sun},C_4,C_5,C_6)$-free graphs. However, we next show that the same conclusion holds for this class as well. 
\begin{corollary}
If $G$ is a strongly chordal graph, we have $\gamma_{\os}(\mathscr{P}_3(G))=\gamma_{\rom}(G)=\we(G^2)$. 
\end{corollary}
\begin{proof}
Indeed, the graph $B_e(G)$ is chordal bipartite if and only if $G$ is a strongly chordal graph (compare to~\citep[Theorem $2.3$ $(b)$]{B-bipclasses}). However, this in turn forces that $\mathscr{P}_3(G)$ is a weakly chordal poset if and only if $G$ is a strongly chordal graph.\footnote{Note that if necessary, we have the freedom of replacing $\mathscr{P}_3(G)$ with the poset $\mathscr{P}_3(G)^m$ by Observation~\ref{obs:up-down}.} Thus, $\mathscr{P}_3(G)$ is a complete  Helly poset. So, the claim follows from Theorem~\ref{thm:comekr-wclique}.
\end{proof}

\begin{remark}
We note that  the inequality $\chi(\overline{G^2})\leq \gamma(G)$ holds for every graph $G$, and it could even be strict in general. However, we have not been able to decide whether it is strict on the class of $(3\text{-sun},C_4,C_5,C_6)$-free graphs.
\end{remark}

\section{Biclique vertex partition in graphs}\label{sec:biclique}
In this section, we prove that the order-sensitive domination number of a poset in the class $\PE_3(k)$ with $k\geq 3$ can be interpreted as the biclique vertex-partition number of a bipartite graph constructed from the poset itself. The idea of associating a bipartite graph to a given poset seems to first appear in the work of Ford and Fulkerson~\citep[page $62$]{FF-book} (which was further studied by Eschen et al.~\citep{EHSS}). The following is a slightly modified version of the bipartite transformation that they consider.

\begin{definition}
Let $P=(X,\leq)$ be a poset. Its \emph{bipartite transformation} $\mathcal{B}(P)$ is the bipartite graph defined by $V(\mathcal{B}(P)):=V_1\cup V_2$ where
$V_1:=\{x'\colon x\in \Min(P)\cup \Mid(P)\}$ and $V_2:=\{x''\colon x\in \Max(P)\cup \Mid(P)\}$ such that $x'y''\in E(\mathcal{B}(P))$ if and only if either $x<y$ in $P$ or $x=y\in \Mid(P)$.
\end{definition}
\begin{example}
We draw a poset and its bipartite transformation in Figure~\ref{fig:fig1}. Observe that $\bp(\mathcal{B}(P))=2<3=\gamma(\mathcal{B}(P))$.

\begin{figure}[ht]
\centering     
\subfigure{
\begin{tikzpicture}[scale=.7]
\node [nodel] at (0,0) (v1) [label=left:$1$]{};
\node [nodel] at (0,2) (v2) [label=left:$2$]{}
               edge [line width=0.5mm] (v1);
\node [nodel] at (0,4) (v3) [label=left:$3$]{}
               edge [line width=0.5mm] (v2);
\node [nodel] at (2,0) (v4) [label=right:$4$]{};
\node [nodel] at (2,2) (v5) [label=right:$5$]{}
           edge [line width=0.5mm] (v4);
\node [nodel] at (2,4) (v6) [label=right:$6$]{}
           edge [line width=0.5mm] (v5)
           edge [line width=0.5mm] (v2);
\node at (1,-.2) (v) [label=below:$P$] {};
\end{tikzpicture}
}
\hspace*{2cm}
\subfigure{
\begin{tikzpicture}[scale=0.7]
\node [nod3] at (0,0) (u1) [label=below:$2'$]{};
\node [nod3] at (2,0) (u2) [label=below:$1'$]{};
\node [nod3] at (4,0) (u3) [label=below:$4'$]{};
\node [nod3] at (6,0) (u4) [label=below:$5'$]{};
\node [nod3] at (0,2) (u5) [label=above:$2''$]{}
           edge [] (u1)
           edge [] (u2);
\node [nod3] at (2,2) (u6) [label=above:$3''$]{}
          edge [] (u1)
          edge [] (u2);
\node [nod3] at (4,2) (u7) [label=above:$6''$] {}
           edge [] (u1)
           edge [] (u2)
           edge [] (u3)
           edge [] (u4);
\node [nod3] at (6,2) (u8) [label=above:$5''$] {}
           edge [] (u3)
           edge [] (u4);           

\node at (3,-1) (u) [label=below:$\mathcal{B}(P)$] {};
\end{tikzpicture}   
}
\caption{A poset $P$ and its bipartite transformation.}
\label{fig:fig1}
\end{figure}
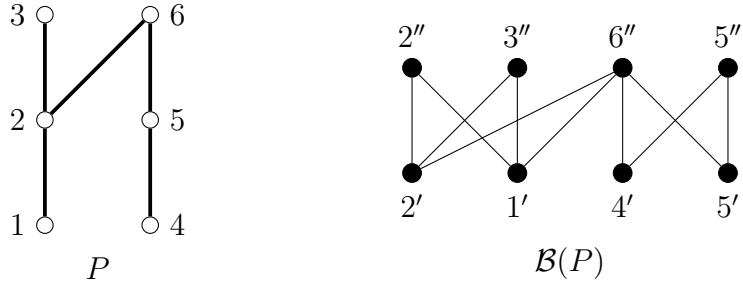
\end{example}
The proof of the following is almost identical to that of~\citep[Theorem $1$]{EHSS}, which we include for the sake of completeness.
\begin{theorem}\label{thm:wcp-cbp}
$\comp(P)$ is a weakly chordal graph if and only if $\mathcal{B}(P)$ is a chordal bipartite graph.
\end{theorem}
\begin{proof}
Suppose that $\comp(P)$ is a weakly chordal graph where $P=(X,\leq)$. Assume for a contradiction that $C$ is a chordless cycle of lenght at least six in $\mathcal{B}(P)$. We first note that $|C\cap\{x', x''\}|\leq 1$ for any element $x$, since otherwise the existence of edges $x'y''$ and $w'x''$ in $C$ would imply the chord $w'y''$ in $\mathcal{B}(P)$ such that $\mathcal{B}(P)[\{w',x',x'',y''\}]\cong C_4 $, that is, every cycle containing $x',x''$ has a chord. Secondly, $xy\notin E(\comp(P))$ for any two vertices $x',y' \in C$, since a comparability $x<y$ in $P$ would imply that $N_{\mathcal{B}(P)}(x')\subseteq N_{\mathcal{B}(P)}(y'')$, which is again followed by a chord in $C$. Similarly, if $x'',y'' \in C$, then $xy\notin E(\comp(P))$. Finally, if any two vertices $x'$ and $y''$ in $C$ are nonadjacent, then $x$ and $y$ are incomparable in $P$. Indeed, if $y<x$ in $P$, then the neighbors of $x'$ in $C$ together with the neighbors of $y''$ in $C$ would induce a $C_4$ in $\mathcal{B}(P)$. Therefore any chordless cycle of lenght at least six in $\mathcal{B}(P)$ corresponds to an induced cycle of the same length in $\comp(P)$.

Conversely, assume that $\comp(P)$ has a chordless cycle $F=\{x_1,x_2,\ldots, x_{2k}\}$ for some $k\geq 3$. Assume without loss of generality that $x_1<x_2$ in $P$. Note that $F$ does not contain any vertex $x$ such that $zx,xy\in E(F)$ and $z<x<y$ in $P$, since otherwise the transitivity introduces the $zy$ chord. It follows that the comparabilities between the elements of $F$ is in the fashion $x_1<x_2>x_3<x_4>\ldots<x_{2k}>x_1$ in $P$, which yields $\mathcal{B}(P)[\{x_1',x_2'',\ldots, x_{2k-1}', x_{2k}''\}]\cong C_{2k}$, a contradiction.
\end{proof}

For a given $a\in \Mid(P)$, we define $U_a:=\{x''\in V_2\colon x\in U_P[a]\}$ and $D_a:=\{x'\in V_1\colon x\in D_P[a]\}$. Similarly, if $a\in \Max(P)$ or $a\in \Min(P)$, we set $B_a:=\st(a'')$ and $B_a:=\st(a')$, respectively, where $\st(v)$ denotes the subgraph isomorphic to the star $K_{1,n}$ induced by the vertex $v$ and its neighbors in $\mathcal{B}(P)$.

\begin{observation}
If $a\in \Mid(P)$, then $B_a:=D_a\cup U_a$ is a biclique in $\mathcal{B}(P)$.
\end{observation}

\begin{definition}
A biclique $B$ in $\mathcal{B}(P)$ is said to be a \emph{star biclique} if $B=B_x$ for some $x\in \Min(P)\cup \Max(P)$. Similarly, a biclique $B$ is called a \emph{double-star biclique} if $B=B_x$ for some $x\in \Mid(P)$.
\end{definition}

For a given subset $U\subseteq V(\mathcal{B}(P))$, we define its \emph{core} by
$$\core(U):=\{u\in X\colon \textnormal{either}\;u'\in U\;\textnormal{or}\;u''\in U\}.$$

\begin{theorem}\label{thm:bp-ms-gamma}
$\bp(\mathcal{B}(P))=\gamma_{\os}(P)$ for any poset $P\in \PE_{3}(k)$ with $k\geq 3$.
\end{theorem}
\begin{proof}  
Let $D$ be an order-sensitive dominating set for $P$. Consider the family 
$$\SE(D):=\{B_a\colon a\in D\}$$ of bicliques in $\mathcal{B}(P)$. We claim that $\SE(D)$ is a biclique vertex-covering of $\mathcal{B}(P)$. It is sufficient to show that for every $x\in P$, the vertex $x'\in \mathcal{B}(P) $ belongs to a biclique in $\SE(D)$. Note that for the vertex $x'' \in \mathcal{B}(P)$, the proof will be similar. We may first assume that $x\in \Min(P)$ and $B_x\notin \SE(D)$. Choose an element $w\in U_P(x)\cap \Mid(P)$. Since $D$ is an os-dominating set, we have either $w\in D$ or there exist an element $z\in D\cap U_P(w)$. If $w\in D$, then $x'\in B_w\in \SE(D)$ or else $x'\in B_z\in \SE(D)$ for some $z\in D\cap U_P(w)$. Therefore, we may further assume that $x\in \Mid(P)$. If $x\in D$, then we clearly have  $x'\in B_x\in \SE(D)$. If $x\notin D$, then there exist elements $y\in D\cap D_P(x)$ and $z \in D\cap U_P(x)$, since $D$ is an os-dominating set. In such a case $x'$ belongs to both of the bicliques $B_y, B_z \in \SE(D)$.

Now, we are only left to verify that $\gamma_{\os}(P)\leq \bp(\mathcal{B}(P))$. In order to prove that we will show that there exists a biclique vertex-cover of $\mathcal{B}(P)$ of size $\bp(\mathcal{B}(P))=k$ consisting only star and double-star bicliques in $\mathcal{B}(P)$.
Assume otherwise that this is not possible. Choose a biclique vertex-cover $\BE$ of $\mathcal{B}(P)$ among all biclique vertex-covers of size $k$ such that $\BE$ contains the maximum number of star and double-star bicliques.
We write $\BE=\BE_1\cup \BE_2$ such that $\BE_1$ consists of all star and double-star bicliques in $\BE$ and $\BE_2\neq \emptyset$. If $B\in \BE$, 
we set $B=U(B)_1\cup U(B)_2$, where $U(B)_t\subseteq V_t$ for each $t=1,2$. Now, given a biclique $B\in \BE_2$.

{\bf Claim:} There exists no $a\in \Mid(P)$ such that $a',a''\in B$. 

{\it Proof of the Claim:} If $a',a''\in B$, then
$\core(U(B)_1)\subseteq D_P[a]$ and $\core(U(B)_2)\subseteq U_P[a]$ in $P$. In other words, we would have $B\subseteq B_a$ so that we may replace $B$ with $B_a$ in $\BE$ increasing the number of star and double-star bicliques, which is not possible.

We may therefore assume that if $B\in \BE_2$, then $|B\cap \{a',a''\}|\leq 1$ for every $a\in \Mid(P)$.

{\bf Case $1$.} There exists a biclique $B\in \BE_2$ and $u\in \Mid(P)$ such that $|B\cap \{u',u''\}|=1$.

Assume without loss of generality that $u'\in B$ and $u''\notin B$. Since $\BE$ is a biclique vertex-cover, there exists $H\in \BE$ such that $u''\in H$.

{\bf Subcase $1.1$.} $H\in \BE_1$. Suppose first that $H$ is a star biclique. In other words, there exists $x\in \Min(P)$ such that $H=B_x=\st(x)$. It then follows that $\core(U(B)_2)\subseteq U_P(x)$ holds in $P$. Indeed, if $w\in \core(U(B)_2)$, then $u<w$, and since $x<u$, we have $x<w$ by the transitivity. Now, if $z\in \core(U(B)_2)\cap \Max(P)$, we define $B':=B_z=\st(z)$ so that $B'$ is a star biclique in $\mathcal{B}(P)$ satisfying $B\cup H\subseteq B'\cup H$. Thus, we may replace $B$ in $\BE$ with the star biclique $B'$ preserving the vertex-covering property, a contradiction.

Secondly, let $H$ be a double-star biclique, that is, $H=B_y$ for some $y\in \Mid(P)$. Once again, we have that $\core(U(B)_2)\subseteq U_P(y)$. 
Similar to above case, we may replace $B$ with a star biclique $B_z$ for some $z\in \core(U(B)_2)\cap \Max(P)$, which is not possible.

{\bf Subcase $1.2$.} $H\in \BE_2$. Since $u'\in B$ and $u''\in H$, we have that $\core(U(B)_2)\subseteq CU_P(\core(U(H)_1))$. Therefore, if we define
$B':=B_z$ for some $z\in \core(U(B)_2)\cap \Max(P)$ and $H':=U(H)_1\cup (U(H)_2\cup U(B)_2)$, it follows that both $B'$ and $H'$ are bicliques in $\mathcal{B}(P)$ satisfying $B\cup H\subseteq B'\cup H'$. Moreover, the family
$(\BE\setminus \{B,H\})\cup \{B',H'\}$ is a biclique vertex-cover of $\mathcal{B}(P)$ whose number of stars and double stars is strictly greater then that of $\BE$, a contradiction.

{\bf Case $2$.} If $B\in \BE_2$, then $B\cap \{v',v''\}=\emptyset$ for every $v\in \Mid(P)$. It then follows that $U(B)_1\subseteq \Min(P)$ and $U(B)_2\subseteq \Max(P)$. 

{\bf Subcase $2.1$.} There exists $q\in \Mid(P)$ such that $q\in CU_P(\core(U(B)_1))$. If we define $B':=U(B)_1\cup (U(B)_2\cup \{q''\})$, then $B'$ is a biclique in $\mathcal{B}(P)$. Now, if we set $\BE':=(\BE\setminus \{B\})\cup \{B'\}$, then $\BE'$ is a biclique vertex-cover of $\mathcal{B}(P)$ having the same number of star and double-star bicliques as with that of $\BE$. However, this is not possible by Case $1$.

{\bf Subcase $2.2$.} There exists $q\in \Mid(P)$ such that $q\in CD_P(\core(U(B)_2))$. This subcase can be treated as in Subcase $2.1$.

{\bf Subcase $2.3$.} $CU_P(\core(U(B)_1))\cap \Mid(P)=\emptyset$ and $CD_P(\core(U(B)_2))\cap \Mid(P)=\emptyset$. In such a case, we claim that $\BE$ can not be a minimal biclique vertex-cover.

We define $M(B)$ to be the subset of $\Mid(P)$ such that if $q\in M(B)$, then there exists $x_q\in U(B)_1$ so that $x_q<q$ or (there exists) $y_q\in U(B)_2$ so that $q<y_q$ hold in $P$. Observe that $|M(B)|\geq 2$. We partition $M(B)$ into two disjoint subsets as $M(B)=M_1(B)\cup M_2(B)$ in a way that if $q\in M_1(B)$, then $B_q\in \BE$, while $p\in M_2(B)$, then there exist two bicliques
$B'(p)$ and $B''(p)$ in $\BE$ containing the vertices $p'$ and $p''$ respectively. Note that if $q\in M_1(B)$, then $x_q,y_q\in B_q$. 
On the other hand, if $p\in M_2(B)$, then we define
$B_1(p):=(U(B'(p))_1\cup \{x_p\})\cup U(B'(p))_2$ and $B_2(p):=U(B''(p))_1\cup (U(B''(p))_2\cup \{y_p\})$. Observe that both $B_1(p)$ and $B_2(p)$ are bicliques in $\mathcal{B}(P)$. Furthermore, the inclusion

\begin{equation*}
B\subseteq (\bigcup_{q\in M_1(B)} B_q)\cup (\bigcup_{p\in M_2(B)} B_1(p)\cup B_2(p))
\end{equation*}
holds. It then follows that if we replace $B'(p)$ and $B''(p)$ with 
$B_1(p)$ and $B_2(p)$ in $\BE$ respectively, the resulting family is a biclique vertex-cover of $\mathcal{B}(P)$ in which the biclique $B$ is redundant.
This proves the claim.

As a consequence, the bipartite graph $\mathcal{B}(P)$ must have a biclique vertex-cover $\BE$ of size $\bp(\mathcal{B}(P))=k$ such that each biclique in $\BE$ is either a star or a double-star biclique. Now, define $D(\BE):=\{x\in X\colon B_x\in \BE\}$. We claim that $D(\BE)$ is an os-dominating set for $\comp(P)$. The fact that $D(\BE)$ is a dominating set simply follows from property that $\BE$ is a biclique vertex-cover. So, let $p\in \Mid(P)$ be given such that $p\notin D(\BE)$, that is, $B_p\notin \BE$. However, since $\BE$ is a biclique vertex-cover, there exist
$c,d\in X$ so that $p'\in B_c$ and $p''\in B_d$. It means that $p\in D_P(c)\cap U_P(d)$. 

Finally, we conclude that $\gamma_{\os}(P)\leq |D(\BE)|=k=\bp(\mathcal{B}(P)))$. This completes the proof. 
\end{proof}

Now, we have the following trivial results as the consequences of Theorems~\ref{thm:os-dom-roman}, \ref{thm:4-gammaos-gamma} and~\ref{thm:bp-ms-gamma}.

\begin{corollary} 
$\bp(\mathcal{B}(\mathscr{P}_3(G)))=\gamma_{\rom}(G)$ for every connected graph $G$ with order at least two.
\end{corollary}

\begin{corollary} 
$\bp(\mathcal{B}(\mathscr{P}_4(G)))=2\gamma(G)$ for every graph $G$.
\end{corollary}


\section{Complexity of Order-sensitive domination}
We have already enough evidences to conclude that for a given poset $P\in \PE_3(k)$ with $k\in \{3,4\}$ and a positive integer $d$, the problem of deciding whether there exists an os-dominating set in $P$ of size at most $d$ is NP-complete. Before extending this result to posets of arbitrary height, we state the problem more formally.

\begin{align*}
&\textnormal{OS-DOMINATING SET:}\\
&\textnormal{{\it Instance}: A  poset}\;P\in \PE_{3}(k)\;\textnormal{ and a positive integer}\;d.\\
&\textnormal{{\it Question}:  Does}\;P\;\textnormal{have an order-sensitive dominating set of size at most}\;d?
\end{align*}

\begin{corollary}\label{cor:even-complexity}
\textnormal{OS-DOMINATING SET} problem in $\PE_3(k)$ for $3\leq k\leq 4$  is \textnormal{NP}-complete.
\end{corollary}

\begin{proof}
It is well-known that DOMINATING SET problem in graphs is NP-complete. Hence, the result follows from Theorem \ref{thm:4-gammaos-gamma} together with Proposition~\ref{prop:poset-red3}. 
\end{proof}

For completeness, we next consider OS-DOMINATING SET problem for posets of arbitrary height.  We will show that the problem remains NP-complete by using a reduction from Equal 3-Satisfiability (EQUAL-$3$-SAT) problem. EQUAL-$3$-SAT is a special case of $3$-SAT problem where the formula restricted to the property that the number of clauses equal to the number of variables. In \cite{juho}, it has been proved that EQUAL-3-SAT is NP-complete.

\begin{theorem}\label{thm.odd-complexity}
\textnormal{OS-DOMINATING SET} problem is \textnormal{NP}-complete for posets in $\PE_3(k)$ for every $k\geq 4$.
\end{theorem}

\begin{proof}
Given a poset $P\in \PE_3(k)$ and a certificate $S$, we can easily verify that $S$ is an os-dominating set or not. Thus, the decision problem is in NP. 

We use a reduction from EQUAL-$3$-SAT to the problem. Given an instance \V for EQUAL-$3$-SAT problem with $n$ variables and $n$ clauses, we will construct a poset $P\in \PE_3(k)$ with $n^2(k-2)+4n$ elements.

Let $\CE=\{C_1,C_2,\ldots,C_n\}$ be $n$ clauses of \V and $C_i=(\ell_{i1},\ell_{i2},\ell_{i3})$ for $i\in [n]$. We construct the (graded and self-dual) poset $P$ from \V as follows: For every boolean variable $x$ of \V, we associate four elements $a,b$ and $a',b'$, where $x'$ stands for the negation of the  variable $x$. We set $\Max(P):=\{a_1,a_1',\ldots, a_n,a_n'\}$ and $\Min(P):=\{b_1,b_1',\ldots, b_n,b_n'\}$. In addition, for every clause $C_i=(\ell_{i1},\ell_{i2},\ell_{i3})$ for $i\in [n]$,  we associate $n$ disjoint $(k-2)$-chains, $T_{i1},\ldots,T_{in}$. Then, we regroup the resulting $n^2$ disjoint $(k-2)$-chains as $H_p=\{T_{1p},\ldots,T_{np}\}$ so as to set $\Mid(P):=H(\CE)=H_1\cup H_2\cup\ldots\cup H_n$. In particular, we label the $l^{\textnormal{th}}$-layer of $H_p$ by
$c^l_{1p},\ldots,c^l_{np}$ for each $p\in [n]$.
For the sake of simplicity, we denote by $L_1,\ldots,L_{k-2}$, the layers of $H(\CE)$. 

Now, if a literal $\ell_{ij}$ occurs in a clause $C_i$ for $i\in [n]$, $j\in [3]$, we let $c^{k-2}_{ip}$ be covered by the element in $\Max(P)$ corresponding to $\ell_{ij}$ for each $p\in [n]$. Symmetrically, we let $c^1_{ip}$ covers the element in $\Min(P)$ corresponding to $\ell_{ij}$ for each $p\in [n]$. This completes the construction of $P$ (see Figure \ref{fig:equal} for an illustration).

Denote by $U_{C_i}$ and $D_{C_i}$ the elements in $\Max(P)$ and $\Min(P)$ respectively,  corresponding to $\ell_{i1},\ell_{i2},\ell_{i3}$.

We claim that \V  is satisfiable if and only if $P$ has an os-dominating set of size at most $2n$.

\begin{figure}[ht]
\centering  
\includegraphics[width=.8\textwidth]{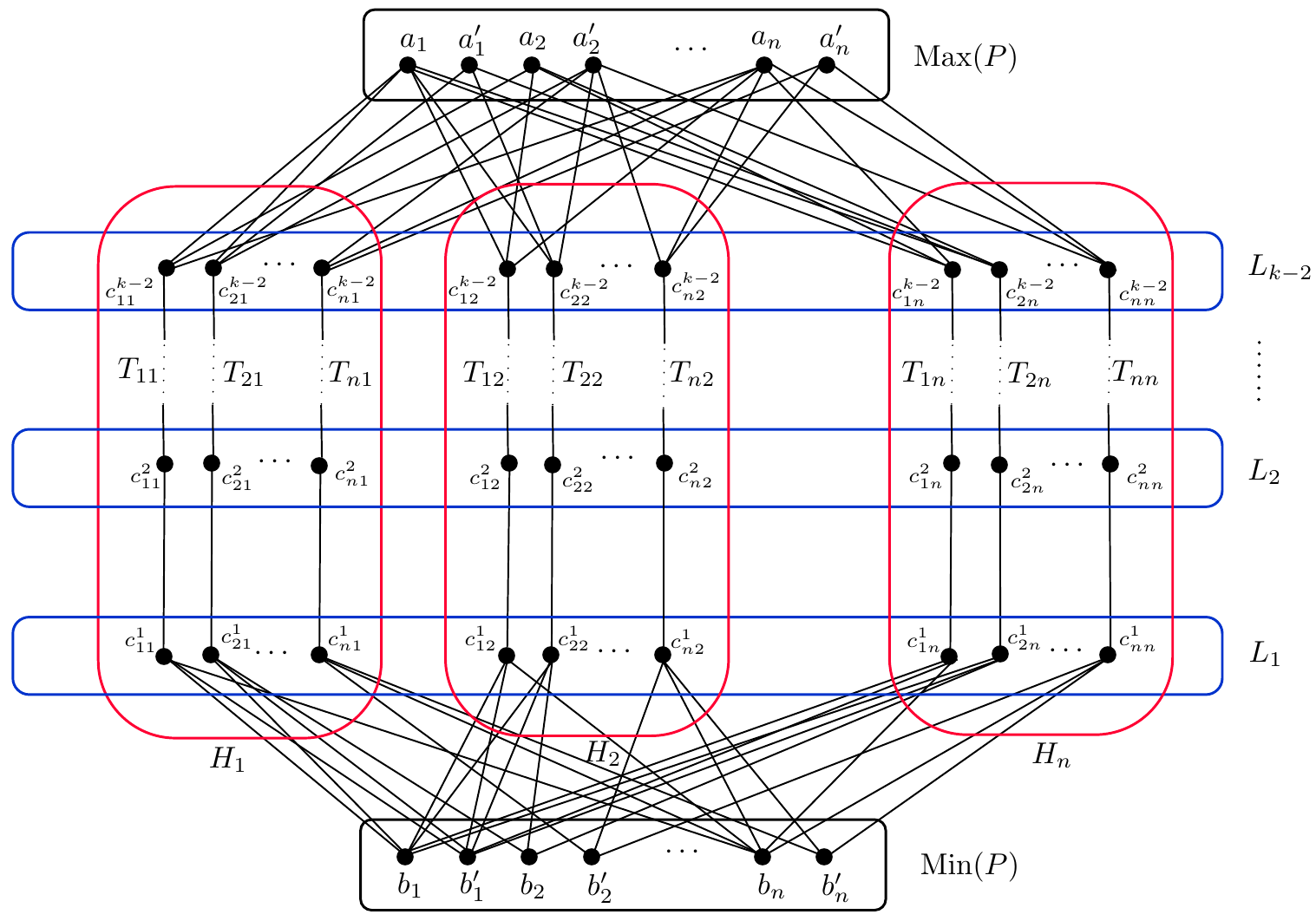}
\caption{An illustration of a poset constructed in Theorem~\ref{thm.odd-complexity}.}
\label{fig:equal}
\end{figure} 

Assume \C is satisfiable by a truth assignment. Then, each of $n$ clauses of \V has at least one literal assigned TRUE. Let $S$ be a set of elements corresponding to exactly one literal in each clause of \V assigned TRUE. Obviously, we have $|S|=n$. If we denote by $S_{\max}$ and $S_{\min}$, the subsets of $\Max(P)$ and $\Min(P)$ respectively, corresponding to those elements of $S$, then $S_{\max}\cup S_{\min}$ provides an os-dominating set for $P$ with size $2n$. 

Conversely, we now assume that $P$ has an os-dominating set $S$ with $|S| \leq 2n$. Without loss of generality, we may assume that $S$ is of minimum order. We claim that we can always find such a set satisfying
$S\cap H(\CE)=\emptyset$. Suppose that this is not the case, and let $S$ be such a set containing fewest elements from $H(\CE)$.
So, there exists $p\in [n]$ such that $c^l_{ip}\in H_p\cap S$ for some $1\leq l\leq k-2$. If $l=1$, we consider the element $c^{k-2}_{ip}$ in the $(k-2)$-chain $T_{ip}$. Then, either $c^{k-2}_{ip}\in S$ or else $U_P(c^{k-2}_{ip})\cap S\neq \emptyset$. If $c^{k-2}_{ip}\in S$, choose
$a_i\in \Max(P)$ and $b_i\in \Min(P)$ such that $a_i$ covers $c^{k-2}_{ip}$ and $b_i$ is covered by $c^{1}_{ip}$ in $P$. Then, 
the set
$(S\setminus \{c^{1}_{ip}, c^{k-2}_{ip}\})\cup \{a_i,b_i\}$ is an os-dominating set containing fewer elements from $H(\CE)$, a contradiction.
On the other hand, if $c^{k-2}_{ip}\notin S$, that is, $U_P(c^{k-2}_{ip})\cap S\neq \emptyset$, then the set
$(S\setminus \{c^1_{ip}\})\cup \{b_i\}$ is still an os-dominating set for $P$, a contradiction. A similar argument applies if $l>1$ by considering the element $c^{1}_{ip}$ in the $(k-2)$-chain $T_{ip}$. This proves the claim.

Now, let $S$ be an os-dominating set with $|S| \leq 2n$ and $S\cap H(\CE)=\emptyset$. It then follows that every clause $C_i$ has a literal whose corresponding element is in $S$.  Thus, the set $S$  corresponds to literals of \V assigned TRUE. In other words, at least one literal corresponding to an element of $U_{C_i}$ (resp. of $D_{C_i}$) is assigned TRUE for each $i\in [n]$.  Hence, \V is  satisfiable.
\end{proof}

Now, combining Corollary~\ref{cor:even-complexity} and Theorem~\ref{thm.odd-complexity}, we have a complete resolution:

\begin{corollary}
\textnormal{OS-DOMINATING SET} problem is \textnormal{NP}-complete for posets in $\PE_3(k)$ for every $k\geq 3$.
\end{corollary}

What we have not been able to resolve is the complexity of OS-DOMINATING SET problem in the subfamily of $\PE(3)$  consisting of weakly chordal posets.

\begin{problem}\label{prob:prob1}
Determine the complexity of OS-DOMINATING SET problem on weakly chordal posets in $\PE(3)$?
\end{problem} 

Note that Problem~\ref{prob:prob1} is of particular interest due to Theorems~\ref{thm:wcp-cbp} and~\ref{thm:bp-ms-gamma}. In other words, if the outcome of Problem~\ref{prob:prob1} turns out to be NP-complete, it then follows that the BICLIUE VERTEX-PARTITION problem for chordal bipartite graphs is NP-complete, that resolves an open problem of Duginov~\citep{ODug}. 

\section{Further comments}
In this section, we offer a short discussion on some possible new directions where to lead from here.

Regarding to the results of Section~\ref{sect:os-dom-graphs}, the most prominent question is to understand whether Theorems~\ref{thm:os-dom-roman} and ~\ref{thm:4-gammaos-gamma} may have any role of obtaining tight bounds on the Roman domination and domination numbers of graphs. In particular, recall that there are various conjectured upper bounds on the domination number of graphs~\citep{Henning2016}. 

Consider a graph $G=(V,E)$ such that $\mathscr{P}_4(G)$ is a complete  Helly poset. By Lemma~\ref{lem:red-ekr}, it follows that $P(G):=\red(\mathscr{P}_4(G))$ is a complete  Helly poset as well. Furthermore, if we write $V_i:=\{v_i\colon v\in V\}$ for $i=1,2$, the middle graph of $P(G)$ is the graph on $V_1\cup V_2$ with the following properties:\medskip
\begin{itemize}
\item[•] $\midd(P(G))[V_i]\cong G^2$ for $i=1,2$,\\
\item[•] $u_1v_2\in E(\midd(P(G)))$ if and only if $\dist_G(u,v)\leq 3$ for every pair (not necessarily distinct) of vertices $u,v\in V$. 
\end{itemize}
Therefore, a possible optimal clique partition of $\midd(P(G))$ gives rise to an upper bound on the domination number of the underlying graph. In this guise, a further study on the coloring of complements of middle graphs of  Helly posets is required.\medskip

Recall that when $P_1=(X_1,\leq_1)$ and $P_2=(X_2,\leq_2)$ are two posets, their \emph{Cartesian product} $P_1\times P_2$ is defined to be the poset on $X_1\times X_2$ such that $(x_1,x_2)\leq (y_1,y_2)$ if and only if $x_1\leq_1 y_1$ and $x_2\leq_2 y_2$.

\begin{problem}\label{prob:product}
Is it true that the inequality $\gamma_{\os}(P_1\times P_2)\geq \gamma_{\os}(P_1)\gamma_{\os}(P_2)$ holds for posets $P_1,P_2\in \PE_2(k)$ with $k\geq 2$ such that at least one of the posets $P_1$ and $P_2$ has neither a maximum nor a minimum element?
\end{problem}

Problem~\ref{prob:product} is still interesting, even in the case of bipartite graphs considered as posets of height two.

\begin{problem}\label{prob:product-bip}
If $B_i$ is a bipartite graph without any isolated vertex for $i=1,2$,
is it true that the inequality $\gamma_{\os}(B_1\times B_2)\geq \gamma_{\os}(B_1)\gamma_{\os}(B_2)$ holds?
\end{problem}
We remark that $B_1\times B_2$ is a graded poset of height three, and $\gamma_{\os}(B)=\gamma(B)$ for any bipartite graph $B$ by our earlier convention. Thus, Problem~\ref{prob:product-bip} asks for the validity of the inequality  $\gamma_{\os}(B_1\times B_2)\geq \gamma(B_1)\gamma(B_2)$.

When one of the bipartite graphs is a single edge, that is, $B_2=K_2$, we have an affirmative answer to Problem~\ref{prob:product-bip}.

\begin{lemma}
$\gamma_{\os}(B\times K_2)\geq \gamma_t(B)\geq \gamma(B)$ for every connected bipartite graph $B$ with at least three vertices.
\end{lemma}
\begin{proof}
We write $V(B)=X\cup Y$ and $V(K_2)=\{1,2\}$, and initially verify that there exists an os-dominating set $D$ in $B\times K_2$ of minimum order such that $D\cap \Mid(B\times K_2)=\emptyset$. Suppose otherwise this is not possible, and choose an os-dominating set $D$ of minimum order containing fewest mid-elements of $B\times K_2$. Assume that $(x,2)\in D\cap \Mid(B\times K_2)$ for some $x\in X$. Let $y\in N_B(x)$ be a neighbor of $x$. If $(y,1)\in D\cap \Mid(B\times K_2)$, then the set 
$$D':=(D\setminus \{(x,2),(y,1)\})\cup \{(x,1),(y,2)\}$$ 
is still an os-dominating set for $B\times K_2$, containing fewer mid-elements than $D$, a contradiction. Thus, we must have $(y,1)\notin D$ for every vertex $y\in N_B(x)$. Now, since $D$ is an os-dominating set,
there exist $(u,1),(v,2)\in D$ such that $(u,1)<(y,1)<(v,2)$ in $B\times K_2$. The only possibility for the vertex $v\in Y$ is that $y=v$. In other words, we must have $(y,2)\in D$. However, it then follows that the set $D'':=(D\setminus \{(x,2)\})\cup \{(x,1)\}$ is an os-dominating set containing fewer mid-elements than $D$, a contradiction. This proves the claim.

Now, let $D$ be an os-dominating set for the poset $B\times K_2$ with $|D|=\gamma_{\os}(B\times K_2)$ such that $D\cap \Mid(B\times K_2)=\emptyset$. If we define 
$$T(D):=\{x\in X\colon (x,1)\in D\}\cup \{y\in Y\colon (y,2)\in D\},$$
then $T(D)$ is a total dominating set for $B$ of order $|D|$.
\end{proof}
As a final remark, we note that $\midd(B\times K_2)\cong B$ for every bipartite graph without isolated vertices.

\bibliographystyle{abbrv}  
\bibliography{os-dom}


\end{document}